\documentclass[11pt]{article}
\usepackage{amssymb}
\usepackage{amsmath}
\usepackage{amsthm} % defines environment proof
\usepackage{wrapfig}
\usepackage{times}
\usepackage{color}

\usepackage{graphicx}
\usepackage{float} % needed for pictures
\usepackage{epic}  % needed for \dashline, \drawline \spline etc.

\allowdisplaybreaks[4] % The optional parameter 1-4 sets the intensity
\newtheorem{lemma}{Lemma}%[section]
\newtheorem{theorem}{Theorem}%[section]

\renewenvironment{proof}{ %\removelastskip\par\medskip
\noindent{\bf Proof.} \rm}{\penalty-20\null\hfill $\square$}
\numberwithin{equation}{section}

\newcommand{\R}{{\mathbb R}}

\newcommand{\bfu}{\mathbf{u}}
\newcommand{\bfv}{\mathbf{v}}
\newcommand{\bfw}{\mathbf{w}}
\newcommand{\bfx}{\mathbf{x}}
\newcommand{\bfy}{\mathbf{y}}

\newcommand{\bfL}{\mathbf{L}}

\newcommand{\bfW}{\mathbf{W}}

\newcommand{\cE}{{\cal E}}

\newcommand{\cG}{{\cal G}}

\newcommand{\cP}{{\cal P}}
\newcommand{\cR}{{\cal R}}

\newcommand{\rmd}{\mathrm{d}}
\newcommand{\rme}{\mathrm{e}}

\renewcommand{\div}{\mathrm{div}\,}

\newcommand{\bfzero}{\mathbf{0}}

\newcommand{\esssup}{\mathrm{ess\, sup}}

\newcommand{\br}{\hbox to 0.7pt{}}
\newcommand{\bfxb}{\bfx'}
\newcommand{\bfyb}{\bfy'}
\newcommand{\tb}{t'}

\newcommand{\bfvb}{\bfv'}

\newcommand{\vb}{v'}
\newcommand{\xb}{x'}
\newcommand{\yb}{y'}
\newcommand{\nablab}{\nabla'}
\newcommand{\Deltab}{\Delta'}
\newcommand{\divb}{{\rm div'}\, }
\newcommand{\pb}{p'}
\newcommand{\partialb}{\partial'}
\newcommand{\Uar}{U_{a,\rho}}
\newcommand{\Var}{V_{\hspace{-0.2pt}a,\rho}}
\newcommand{\Qab}{Q'_{a}}
\newcommand{\Uab}{U'_{a}}
\newcommand{\Vab}{V'_{\hspace{-0.2pt}a}}
\newcommand{\Mb}[2]{A'_{{#1}\hspace{-0.2pt},
\hspace{0.5pt}{#2}}}
\newcommand{\Bb}[1]{B'_{{#1}}}
\newcommand{\M}[2]{A_{{#1}\hspace{-0.5pt},\hspace{0.5pt}{#2}}}
\newcommand{\Bbu}[1]{B'{\hspace{-3pt} \lower 3pt\hbox{\tiny
#1}}}
\newcommand{\Mbu}[2]{A'{\hspace{-1.5pt} \lower 3pt\hbox{\tiny
#1\hspace{-2.5pt},\hspace{1.5pt}#2}}}
\newcommand{\Mbuu}[2]{A'{\hspace{-1.5pt} \lower 3pt\hbox{\tiny
#1,\hspace{1.5pt}#2}}}

\newcounter{constants}
\newcommand{\cn}[2]{ \addtocounter{constants}{1}
\newcounter{c#1#2}
\setcounter{c#1#2}{\value{constants}} c_{\arabic{c#1#2}} }
\newcommand{\cc}[2]{c_{\arabic{c#1#2}}}

\definecolor{lightgrey}{rgb}{0.9,0.9,0.9}
\definecolor{modra}{rgb}{0.1,0.1,0.7}
\definecolor{darkgreen}{rgb}{0.1,0.7,0.1}
\definecolor{cervena}{rgb}{0.7,0.1,0.1}
\definecolor{lightgreen}{rgb}{0.7,1.0,0.6}
\definecolor{lightblue}{rgb}{0.65,0.89,1}
\definecolor{darkblue}{rgb}{0.1,0.5,0.7}
\definecolor{orange}{rgb}{1.0,0.8,0.5}
\definecolor{lightorange}{rgb}{1,0.9,0.7}
\definecolor{darkorange}{rgb}{0.7,0.5,0.1}
\definecolor{rose}{rgb}{1.0,0.8,0.8}

\begin{document}

\title{\LARGE \bf A Refinement of the Local Serrin--Type Regularity
Criterion for a Suitable Weak Solution to the Navier--Stokes
Equations}

\author{Ji\v{r}\'{\i} Neustupa}

\date{}

\maketitle

\begin{abstract}
We formulate a new criterion for regularity of a suitable weak
solution $\bfv$ to the Navier--Stokes equations at the space--time
point $(\bfx_0,t_0)$. The criterion imposes a Serrin--type
integrability condition on $\bfv$ only in a backward neighbourhood
of $(\bfx_0,t_0)$, intersected with the exterior of a certain
space--time paraboloid with vertex at point $(\bfx_0,t_0)$. We
make no special assumptions on the solution in the interior of the
paraboloid.
\end{abstract}

\noindent
{\it AMS math.~classification (2000):} \ 35Q30, 76D03, 76D05.


\section{Introduction} \label{S1}

{\bf The Navier--Stokes system.} \ Let $\Omega$ be a domain in
$\R^3$ and $T>0$. Put $Q_T:=\Omega\times(0,T)$. We deal with the
Navier--Stokes system
\begin{align}
\partial_t\bfv+\bfv\cdot\nabla\bfv\ &=\ -\nabla p+\nu\Delta\bfv &&
\mbox{in}\ Q_T, \label{1.1} \\ \noalign{\vskip 2pt}
\div\bfv\ &=\ 0 \qquad && \mbox{in}\ Q_T \label{1.2}
\end{align}
for the unknown velocity $\bfv=(v_1,v_2,v_3)$ and pressure $p$.
Symbol $\nu$ denotes the coefficient of viscosity. It is a
positive constant.

\vspace{4pt} \noindent
{\bf Notation.} \ Vector functions and spaces of vector functions
are denoted by boldface letters. The norm of a scalar- or vector-
or tensor-valued function with components in $L^q(\Omega)$
(respectively $W^{k,q}(\Omega)$) is denoted by $\|\, .\, \|_{q;\,
\Omega}$ (respectively $\|\, .\, \|_{k,q;\, \Omega}$). Norms in
spaces of functions on other domains than $\Omega$ are denoted by
analogy.

\vspace{4pt} \noindent
{\bf Weak and suitable weak solution, regular and singular
points.} \ The definition of a weak solution to the system
(\ref{1.1}), (\ref{1.2}) is explained, together with basic
properties of weak solutions, e.g.~in the books by
O.~A.~Ladyzhenskaya \cite{La}, R.~Temam \cite{Te}, H.~Sohr
\cite{So} and in the survey paper \cite{Ga2} by G.~P.~Galdi. Here,
we only recall that weak solution $\bfv$ of (\ref{1.1}),
(\ref{1.2}) belongs to $L^{\infty}(0,T;\, \bfL^2(\Omega))\cap
L^2(0,T;\, \bfW^{1,2}(\Omega))$. While the existence of a weak
solution (satisfying various kinds of boundary conditions) is
known on an arbitrarily long time interval $(0,T)$, its regularity
is generally an open problem.

By the definition from paper \cite{CKN} by L.~Caffarelli, R.~Kohn
and L.~Nirenberg, the point $(\bfx_0,t_0)\in Q_T$ is said to be a
{\it regular point} of weak solution $\bfv$ if there exists a
neighborhood $U$ of $(\bfx_0,t_0)$ such that
$\bfv\in\bfL^{\infty}(U)$. Points in $Q_T$ that are not regular
are called {\it singular}. A weak solution $\bfv$ of system
(\ref{1.1}), (\ref{1.2}) is called a {\it suitable weak
solution\/} if an associated pressure $p$ belongs to
$L^{5/4}(Q_T)$ and the pair $(\bfv,p)$ satisfies the so called
{\it generalized energy inequality\/}
\begin{equation}
2\nu\int_0^T\int_{\Omega}|\nabla\bfv|^2\, \phi\; \rmd\bfx\,
\rmd\br t\ \leq\ \int_0^T\int_{\Omega}\bigl[\, |\bfv|^2\,
\bigl(\partial_t\phi+\nu\Delta\phi\bigr)+\bigl(|\bfv|^2+2p\bigr)\,
\bfv\cdot\nabla\phi\br\bigr]\; \rmd\bfx\, \rmd\br t \label{1.3*}
\end{equation}
for every non--negative function $\phi$ from $C^{\infty}_0 (Q_T)$.
It is also shown in \cite{CKN} that the set of singular points of
suitable weak solution $\bfv$ has the $1$--dimensional parabolic
measure (which dominates the 1--dimensional Hausdorff measure)
equal to zero. This result follows (by a standard covering
procedure) from the local regularity criterion (also proven in
\cite{CKN}), saying that there exists a universal constant
$\epsilon>0$ such that if
\begin{displaymath}
{\displaystyle\mathrel{\mathop{\lim\, \sup}_{\delta\to 0+}}}\
\frac{1}{\delta}\int_{t_0-7\delta^2/8}^{t_0+\delta^2/8}
\int_{B_{\delta}(\bfx_0)}|\nabla\bfv|^2\; \rmd\bfx\, \rmd\br t\
\leq\ \epsilon
\end{displaymath}
then $(\bfx_0,t_0)$ is a regular point of $\bfv$. Analogous
results and some generalizations can also be found in papers
\cite{LaSe}, \cite{Lin}, \cite{KS}, \cite{Va}, and others.

\vspace{4pt} \noindent
{\bf Some other local regularity criteria.} \ The next criteria
are often called $\epsilon$--criteria because they state that
there exists a universal constant $\epsilon>0$ (sufficiently
small) such that if a certain quantity is less than or equal to
$\epsilon$ then $(\bfx_0,t_0)$ is a regular point of solution
$\bfv$. Constant $\epsilon$ is generally different in different
criteria.

F.~Lin \cite{Lin} proved that the condition
\begin{displaymath}
\lim_{\delta\to 0+}\ \frac{1}{\delta^2}\int_{t_0-\delta^2}^{t_0}
\int_{B_{\delta}(\bfx_0)}\bigl(|\bfv|^3+|p|^{\frac{3}{2}}\bigr)\;
\rmd\bfx\, \rmd\br t\ \leq\ \epsilon
\end{displaymath}
guarantees that $\bfv$ is H\"older continuous in the set
$\overline{B_{\rho}(\bfx_0)}\times[t_0-\rho^2,t_0]$ (for some
$\rho>0$), which implies that $(\bfx_0,t_0)$ is a regular point of
solution $\bfv$. Lin's criterion has been several times improved
(see \cite{NeNe}, \cite{SeSv} and \cite{Wo2}). Wolf's criterion
(see \cite{Wo2}) says that if $3\leq r,s\leq\infty$ and
\begin{displaymath}
\delta^{s\br\left[1-\left(\frac{2}{r}+\frac{3}{s}\right)\right]}
\int_{t_0-\delta^2}^{t_0} \biggl(
\int_{B_{\delta}(\bfx_0)}|\bfv|^s\; \rmd\bfx\biggr)^{\!
\frac{r}{s}}\, \rmd\br t\ \leq\ \epsilon
\end{displaymath}
holds for at least one $\delta>0$ then $\bfv$ is H\"older
continuous in the set
$\overline{B_{\delta/2}(\bfx_0)}\times[t_0-\delta^2/4,t_0]$.
Particularly, if we choose $r=s=3$ then we observe that if the
inequality \begin{equation}
\frac{1}{\delta^2}\int_{t_0-\delta^2}^{t_0}
\int_{B_{\delta}(\bfx_0)}|\bfv|^3\; \rmd\bfx\, \rmd\br t\ \leq\
\epsilon \label{1.4*}
\end{equation}
holds for at least one $\delta>0$ then $(\bfx_0,t_0)$ is a regular
point of solution $\bfv$.

A series of other local regularity criteria can be found in
\cite{Ser}, \cite{Ne2}, \cite{FaKoSo3}, etc.

Let us finally recall that S.~Takahashi \cite{Tak} proved that if
the norm of a weak solution $\bfv$ in $L^r_{w}(t_0-\rho^2,t_0;\,
\bfL^s(B_{\rho}(\bfx_0))$ (where $L^r_{w}$ denotes the weak
$L^r$--space and $2/r+3/s\leq 1$, $3<s\leq\infty$) is less than or
equal to $\epsilon$ then $(\bfx_0,t_0)$ is a regular point of
$\bfv$.

Takahashi's criterion has been modified in paper \cite{Ne3}, where
$\bfv$ is supposed to be integrable with powers $r\in[3,\infty)$
(in time) and $s\in(3,\infty)$ (in space) not necessarily in the
whole backward parabolic neighbourhood $Q_{a,\rho}:=B_{\sqrt{a}
\br\rho}(\bfx_0)\times(t_0-\rho^2,t_0)$, but only in the
intersection of this neighbourhood with the exterior of the
space--time paraboloid
\begin{equation}
P_a: \quad \sqrt{a(t_0-t)}=|\bfx-\bfx_0| \label{1.8*}
\end{equation}
(with vertex at $(\bfx_0,t_0)$). Exponents $r$ and $s$ are
required to satisfy the condition $2/r+3/s<1$, and number $a$ is
supposed to satisfy the inequality $0<a<4\nu$ in \cite{Ne3}.
Moreover, it is also supposed in \cite{Ne3} that there exist real
numbers $R$ and $h$ such that $R>1$, $0<h<R-1$ and
\begin{equation}
\int_{t_0-\rho^2/R^2}^{t_0}\biggl(\int_{(R-h)\br\sqrt{a(t_0-t)}<
|\bfx-\bfx_0|<R\sqrt{a(t_0-t)}}\, |p(\bfx,t)|^{\beta}\;
\rmd\bfx\biggr)^{\! \alpha/\beta}\, \rmd\br t\ <\ \infty
\label{1.5*}
\end{equation}
for $\alpha\in[\frac{r}{r-1},\infty)$ and
$\beta\in(\frac{3}{2},\infty)$, satisfying the inequality
$2/\alpha+3/\beta<2$. Note that the domain of the integral in
(\ref{1.5*}) is the exterior of paraboloid $P_a$, intersected with
neighbourhood $Q_{a,\rho}$.

\vspace{4pt} \noindent
{\bf On the result of this paper.} \ In this paper, we improve the
regularity criterion from \cite{Ne3} especially so that we remove
the assumption on the pressure. Concretely, we show that condition
(\ref{1.5*}) can be omitted. Moreover, in comparison to
\cite{Ne3}, we assume that $2/r+3/s=1$ and we also use a weaker
restriction on parameter $a$ (see Theorem \ref{T1}). This is
enabled by finer estimates in Sections \ref{S3}--\ref{S5}, and by
a different treatment of the term containing the transformed
pressure $\pb$ in Section \ref{S3}, see Lemma \ref{L2}.

Our Theorem \ref{T1} (formulated below) imposes the Serrin--type
condition only on velocity $\bfv$ in an arbitrarily small region
$\Uar$ in $Q_T$, which is defined as follows:
\begin{displaymath}
\Uar\ :=\ \bigl\{\, (\bfx,t)\in Q_T;\ t_0-\rho^2<t<t_0\
\mbox{and}\ \sqrt{a(t_0-t)} <|\bfx-\bfx_0|<\sqrt{a}\br\rho\,
\bigr\}
\end{displaymath}
(for $a>0$ and $0<\rho<\sqrt{t_0}$). In contrast to a series of
other regularity criteria, we make no assumptions on $\bfv$ or $p$
in the interior of paraboloid $P_a$, concretely in set $\Var$
which is the interior of $P_a$, intersected with neighbourhood
$Q_{a,\rho}$.

A generalization of Theorem \ref{T1}, where parameter $a$ does not
appear, is presented in Section \ref{S6}.

\begin{center}
\hspace{65mm}
\setlength{\unitlength}{1.1mm}
\begin{picture}(67,37)
  \put(0,4){\vector(1,0){67}} \put(62,0.5){$\bfx$}
  \put(3,2){\vector(0,1){33}} \put(4.5,32){$t$}
  \dashline[+30]{2}(0,28)(8.5,28)
  \dashline[+30]{2}(47.5,28)(66,28)
  \put(49,30){$t=t_0$}
  \dashline[+30]{2}(0,14)(8.5,14)
  \dashline[+30]{2}(47.5,14)(66,14)
  \put(53,10){$t=t_0-\rho^2$}
  \dashline[+30]{1.6}(28,4)(28,15)
  \dashline[+30]{1.5}(28,21)(28,28)
  \put(26.5,0.5){$\bfx_0$}
  \dashline[+30]{2}(46,4)(46,12.5)
  \dashline[+30]{2}(10,4)(10,12.5)
  \put(40,0.5){$|\bfx-\bfx_0|$}
  \put(40,-4.5){$=\sqrt{a}\br\rho$}
  \color{lightgrey}
  % leva cast
  \put(10,27.85){\line(1,0){16.0}} \put(10,27.70){\line(1,0){15.0}}
  \put(10,27.55){\line(1,0){14.3}} \put(10,27.40){\line(1,0){13.7}}
  \put(10,27.25){\line(1,0){13.2}} \put(10,27.10){\line(1,0){12.7}}
  \put(10,26.95){\line(1,0){12.3}} \put(10,26.80){\line(1,0){11.9}}
  \put(10,26.65){\line(1,0){11.6}} \put(10,26.50){\line(1,0){11.3}}
  \put(10,26.35){\line(1,0){11.0}} \put(10,26.20){\line(1,0){10.7}}
  \put(10,26.05){\line(1,0){10.4}} \put(10,25.90){\line(1,0){10.1}}
  \put(10,25.75){\line(1,0){9.8}} \put(10,25.60){\line(1,0){9.6}}
  \put(10,25.45){\line(1,0){9.3}} \put(10,25.30){\line(1,0){9.1}}
  \put(10,25.15){\line(1,0){8.9}} \put(10,25.00){\line(1,0){8.7}}
  \put(10,24.85){\line(1,0){8.5}} \put(10,24.70){\line(1,0){8.3}}
  \put(10,24.55){\line(1,0){8.1}} \put(10,24.40){\line(1,0){7.9}}
  \put(10,24.25){\line(1,0){7.7}} \put(10,24.10){\line(1,0){7.5}}
  \put(10,23.95){\line(1,0){7.3}} \put(10,23.80){\line(1,0){7.1}}
  \put(10,23.65){\line(1,0){7.0}} \put(10,23.50){\line(1,0){6.8}}
  \put(10,23.35){\line(1,0){6.7}} \put(10,23.20){\line(1,0){6.5}}
  \put(10,23.05){\line(1,0){6.4}} \put(10,22.95){\line(1,0){6.2}}
  \put(10,22.80){\line(1,0){6.05}} \put(10,22.65){\line(1,0){5.95}}
  \put(10,22.50){\line(1,0){5.8}} \put(10,22.35){\line(1,0){5.65}}
  \put(10,22.20){\line(1,0){5.55}} \put(10,22.05){\line(1,0){5.4}}
  \put(10,21.90){\line(1,0){5.3}} \put(10,21.75){\line(1,0){5.1}}
  \put(10,21.60){\line(1,0){5.0}} \put(10,21.45){\line(1,0){4.85}}
  \put(10,21.30){\line(1,0){4.7}} \put(10,21.15){\line(1,0){4.6}}
  \put(10,21.00){\line(1,0){4.5}} \put(10,20.85){\line(1,0){4.35}}
  \put(10,20.70){\line(1,0){4.25}} \put(10,20.55){\line(1,0){4.1}}
  \put(10,20.40){\line(1,0){3.95}} \put(10,20.25){\line(1,0){3.85}}
  \put(10,20.10){\line(1,0){3.75}} \put(10,19.95){\line(1,0){3.65}}
  \put(10,19.80){\line(1,0){3.55}} \put(10,19.65){\line(1,0){3.45}}
  \put(10,19.50){\line(1,0){3.35}} \put(10,19.35){\line(1,0){3.20}}
  \put(10,19.20){\line(1,0){3.1}} \put(10,19.05){\line(1,0){3.0}}
  \put(10,18.90){\line(1,0){2.92}} \put(10,18.75){\line(1,0){2.8}}
  \put(10,18.60){\line(1,0){2.7}} \put(10,18.45){\line(1,0){2.6}}
  \put(10,18.30){\line(1,0){2.5}} \put(10,18.15){\line(1,0){2.4}}
  \put(10,18.00){\line(1,0){2.3}} \put(10,17.85){\line(1,0){2.2}}
  \put(10,17.70){\line(1,0){2.1}} \put(10,17.55){\line(1,0){2.0}}
  \put(10,17.40){\line(1,0){1.9}} \put(10,17.25){\line(1,0){1.8}}
  \put(10,17.10){\line(1,0){1.7}}
  \put(10.1,20){\line(0,-1){5.8}} \put(10.25,20){\line(0,-1){5.5}}
  \put(10.4,20){\line(0,-1){5.25}} \put(10.55,20){\line(0,-1){5.0}}
  \put(10.7,20){\line(0,-1){4.75}} \put(10.85,20){\line(0,-1){4.5}}
  \put(11.0,20){\line(0,-1){4.22}} \put(11.15,20){\line(0,-1){3.95}}
  \put(11.3,20){\line(0,-1){3.7}} \put(11.45,20){\line(0,-1){3.40}}
  \put(11.6,20){\line(0,-1){3.2}}
  % prava cast
  \put(46,27.85){\line(-1,0){16.0}} \put(46,27.70){\line(-1,0){15.0}}
  \put(46,27.55){\line(-1,0){14.3}} \put(46,27.40){\line(-1,0){13.7}}
  \put(46,27.25){\line(-1,0){13.2}} \put(46,27.10){\line(-1,0){12.7}}
  \put(46,26.95){\line(-1,0){12.3}} \put(46,26.80){\line(-1,0){11.9}}
  \put(46,26.65){\line(-1,0){11.6}} \put(46,26.50){\line(-1,0){11.3}}
  \put(46,26.35){\line(-1,0){11.0}} \put(46,26.20){\line(-1,0){10.7}}
  \put(46,26.05){\line(-1,0){10.4}} \put(46,25.90){\line(-1,0){10.1}}
  \put(46,25.75){\line(-1,0){9.8}} \put(46,25.60){\line(-1,0){9.6}}
  \put(46,25.45){\line(-1,0){9.3}} \put(46,25.30){\line(-1,0){9.1}}
  \put(46,25.15){\line(-1,0){8.9}} \put(46,25.00){\line(-1,0){8.7}}
  \put(46,24.85){\line(-1,0){8.5}} \put(46,24.70){\line(-1,0){8.3}}
  \put(46,24.55){\line(-1,0){8.1}} \put(46,24.40){\line(-1,0){7.9}}
  \put(46,24.25){\line(-1,0){7.7}} \put(46,24.10){\line(-1,0){7.5}}
  \put(46,23.95){\line(-1,0){7.3}} \put(46,23.80){\line(-1,0){7.1}}
  \put(46,23.65){\line(-1,0){7.0}} \put(46,23.50){\line(-1,0){6.8}}
  \put(46,23.35){\line(-1,0){6.7}} \put(46,23.20){\line(-1,0){6.5}}
  \put(46,23.05){\line(-1,0){6.4}} \put(46,22.95){\line(-1,0){6.2}}
  \put(46,22.80){\line(-1,0){6.05}} \put(46,22.65){\line(-1,0){5.95}}
  \put(46,22.50){\line(-1,0){5.8}} \put(46,22.35){\line(-1,0){5.65}}
  \put(46,22.20){\line(-1,0){5.55}} \put(46,22.05){\line(-1,0){5.4}}
  \put(46,21.90){\line(-1,0){5.25}} \put(46,21.75){\line(-1,0){5.1}}
  \put(46,21.60){\line(-1,0){4.95}} \put(46,21.45){\line(-1,0){4.8}}
  \put(46,21.30){\line(-1,0){4.65}} \put(46,21.15){\line(-1,0){4.55}}
  \put(46,21.00){\line(-1,0){4.45}} \put(46,20.85){\line(-1,0){4.3}}
  \put(46,20.70){\line(-1,0){4.15}} \put(46,20.55){\line(-1,0){4.05}}
  \put(46,20.40){\line(-1,0){3.9}} \put(46,20.25){\line(-1,0){3.8}}
  \put(46,20.10){\line(-1,0){3.7}} \put(46,19.95){\line(-1,0){3.6}}
  \put(46,19.80){\line(-1,0){3.5}} \put(46,19.65){\line(-1,0){3.4}}
  \put(46,19.50){\line(-1,0){3.3}} \put(46,19.35){\line(-1,0){3.15}}
  \put(46,19.20){\line(-1,0){3.05}} \put(46,19.05){\line(-1,0){2.95}}
  \put(46,18.90){\line(-1,0){2.87}} \put(46,18.75){\line(-1,0){2.75}}
  \put(46,18.60){\line(-1,0){2.65}} \put(46,18.45){\line(-1,0){2.55}}
  \put(46,18.30){\line(-1,0){2.45}} \put(46,18.15){\line(-1,0){2.35}}
  \put(46,18.00){\line(-1,0){2.25}} \put(46,17.85){\line(-1,0){2.15}}
  \put(46,17.70){\line(-1,0){2.05}} \put(46,17.55){\line(-1,0){1.95}}
  \put(46,17.40){\line(-1,0){1.85}} \put(46,17.25){\line(-1,0){1.75}}
  \put(46,17.10){\line(-1,0){1.65}}
  \put(45.95,20){\line(0,-1){5.8}} \put(45.8,20){\line(0,-1){5.5}}
  \put(45.65,20){\line(0,-1){5.2}} \put(45.5,20){\line(0,-1){4.95}}
  \put(45.35,20){\line(0,-1){4.7}} \put(45.2,20){\line(0,-1){4.45}}
  \put(45.05,20){\line(0,-1){4.22}} \put(44.9,20){\line(0,-1){3.95}}
  \put(44.75,20){\line(0,-1){3.65}} \put(44.6,20){\line(0,-1){3.4}}
  \put(44.45,20){\line(0,-1){3.1}}
  \color{black}
  \dottedline{0.8}(43,23)(49,23)
  \put(50,22){$\Uar$}
  \dottedline{0.8}(17,12)(14,19)
  \put(16.5,8.5){$P_a$}
  \qbezier(10,14)(9,12)(7.9,9) \qbezier(46,14)(47,12)(48.1,9)
  \thicklines
  \put(10,14){\line(0,1){14}} \put(46,14){\line(0,1){14}}
  \qbezier(10,14)(17,28)(28,28) \qbezier(46,14)(39,28)(28,28)
  %\dottedline{1.2}(10,14)(7.5,9)
  %\dottedline{1.2}(46,14)(48.5,9)
  \put(10,28){\line(1,0){36}}
  \put(10,14){\line(1,0){36}}
  \put(28,28){\circle*{1}} \put(23,30.5){$(\bfx_0,t_0)$}
  \put(25.2,17.3){$\Var$}
  \put(-17,3){Fig.~1:}
\end{picture}
\end{center}

\vspace{-37mm} \noindent
Sets $\Uar$ and $\Var$ are sketched \\
on Fig.~1.  They are separated \\ by paraboloid $P_a$.

\vspace{30mm}
We denote by $\lambda_S(B_1)$ be the least eigenvalue of the
Dirichlet--Stokes operator in the unit ball $B_1$ in $\R^3$. Note
that the question of how to calculate explicitly the eigenvalues
and eigenfunctions of the Stokes operator in the ball was asked by
O.~A.~Ladyzhenskaya in 2003. It can be deduced from \cite{Sa} that
$\lambda_S(B_1)\leq\mu_1^2(J_{3/2})$, where $\mu_1(J_{3/2})\,
\dot=\, 4.4934$ is the first positive root of the Bessel function
$J_{3/2}$. Using the variational representation of
$\lambda_S(B_1)$ (i.e.~that $\lambda_S(B_1)$ equals the infimum of
$\|\nabla\bfu\|_{2;\, B_1}^2/\|\bfu\|_{2;\, B_1}^2$ over all
non--zero divergence--free functions $\bfu\in\bfW^{1,2}_0(B_1)$)
and the analogous representation of $\lambda_L(B_1)$ (the
principal eigenvalue of the Dirichlet--Laplacian on the unit
ball), we obtain $\lambda_S(B_1)\geq\lambda_L(B_1)$. The latter
equals $\pi^2$, see e.g.~\cite{Sa}.

The main result of this paper says:

\begin{theorem}
\label{T1}
Let $\bfv$ be a suitable weak solution of system (\ref{1.1}),
(\ref{1.2}), $(\bfx_0,t_0)\in Q_T$ and $0<a<4\nu\lambda_S(B_1)$.
Let $\rho>0$ be so small that $Q_{a,\rho}\subset Q_T$. Suppose
that function $\bfv$ satisfies the integrability condition in set
$\Uar$\br:
\begin{equation}
\int_{t_0-\rho^2}^{t_0}\biggl(\int_{\sqrt{a(t_0-t)}<
|\bfx-\bfx_0|<\sqrt{a}\br\rho}|\bfv(\bfx,t)|^s\;
\rmd\bfx\biggr)^{\! \frac{r}{s}}\, \rmd\br t\ <\ \infty
\label{1.6*}
\end{equation}
for some $r$, $s$, satisfying the inequalities
\vspace{-4pt}
\begin{equation}
3\leq r<\infty, \qquad 3<s<\infty, \qquad
\frac{2}{r}+\frac{3}{s}=1. \label{1.7*}
\end{equation}
Then $(\bfx_0,t_0)$ is a regular point of solution $\bfv$.
\end{theorem}

Theorem \ref{T1} shows that if a singularity eventually appears in
a suitable weak solution of the Navier--Stokes system (\ref{1.1}),
(\ref{1.2}) at the point $(\bfx_0,t_0)$ then it cannot develop
only around point $\bfx_0$ itself (i.e.~only in set $\Var$). On
the other hand, ``large'' values of velocity must also be
necessarily transferred to the point $\bfx_0$ from the sides at
times $t<t_0$ with the speed increasing to infinity as $t\to
t_0-$.

%-----------------------------------------------------------------------
\section{Proof of Theorem \ref{T1} -- part I} \label{S2}

{\bf Notation and the used regularity criterion.} \ We denote
\begin{displaymath}
\theta(t)\ :=\ \sqrt{a(t_0-t)} \qquad \mbox{and} \qquad G(\delta)\
:=\ \frac{1}{\delta^2}\int_{t_0-\delta^2}^{t_0}
\int_{|\bfx-\bfx_0|<\sqrt{a}\br\delta}|\bfv|^3\; \rmd\bfx\,
\rmd\br t.
\end{displaymath}
We split $G(\delta)$ to two parts:
\vspace{-8pt}
\begin{displaymath}
G(\delta)\ =\ G^I(\delta)+G^{I\hspace{-0.3pt}I}(\delta),
\end{displaymath}
where
\vspace{-8pt}
\begin{alignat}{3}
& G^I(\delta)\ &&:=\ \frac{1}{\delta^2}\int_{t_0-\delta^2}^{t_0}
\int_{\theta(t)<|\bfx-\bfx_0|<\sqrt{a}\br\delta}|\bfv|^3\;
\rmd\bfx\, \rmd\br t, \label{2.1*} \\ \noalign{\vskip 2pt}
& G^{I\hspace{-0.3pt}I}(\delta)\ &&:=\ \frac{1}{\delta^2}
\int_{t_0-\delta^2}^{t_0} \int_{|\bfx-\bfx_0|<\theta(t)}
|\bfv|^3\; \rmd\bfx\, \rmd\br t. \label{2.2*}
\end{alignat}
We will show that
\vspace{-8pt}
\begin{equation}
{\displaystyle\mathrel{\mathop{\lim\ \inf}_{\delta\to 0+}}}\ \
G(\delta)\ =\ 0. \label{2.3*}
\end{equation}
Since (\ref{2.3*}) implies the validity of condition (\ref{1.4*}),
it also implies that $(\bfx_0,t_0)$ is a regular point of solution
$\bfv$.

\vspace{4pt} \noindent
{\bf An estimate of $G^I(\delta)$.} \ Assume that $r>3$. Then
$G^I(\delta)$ can be estimated as follows:
\begin{align*}
G^I(\delta)\ &\leq\ \frac{1}{\delta^2}\int_{t_0-\delta^2}^{t_0}
\biggl(\int_{\theta(t)<|\bfx-\bfx_0|<\sqrt{a}\br\delta} |\bfv|^s\;
\rmd\bfx\biggr)^{\! \frac{3}{s}}\, \Bigl(\frac{4\pi\br(\sqrt{a}\br
\delta)^3}{3}\Bigr)^{1-\frac{3}{s}}\, \rmd\br t
\nonumber \\
& \leq\ \Bigl[\frac{4\pi
a^{\frac{3}{2}}}{3}\Bigr]^{1-\frac{3}{s}}\
\biggl[\br\int_{t_0-\delta^2}^{t_0}
\biggl(\int_{\theta(t)<|\bfx-\bfx_0|<\sqrt{a}\br\delta} |\bfv|^s\;
\rmd\bfx\biggr)^{\! \frac{r}{s}}\, \rmd t\br\biggr]^{\frac{3}{r}}.
\end{align*}
This implies, due to conditions (\ref{1.6*}) and (\ref{1.7*}),
that
\begin{equation}
\lim_{\delta\to 0+}\ G^I(\delta)\ =\ 0. \label{2.4*}
\end{equation}
We obtain the same information in the case $r=3$, too.

\vspace{4pt} \noindent
{\bf Transformation to the new coordinates $\bfxb,\tb$.} \ In
order to estimate $G^{I\hspace{-0.3pt}I}(\delta)$, we transform
the integral in (\ref{2.2*}) and the system (\ref{1.1}),
(\ref{1.2}) to the new coordinates $\bfxb$ and $\tb$, which are
related to $\bfx$ and $t$ through the formulas
\vspace{-2pt}
\begin{equation}
\bfxb=\frac{\bfx-\bfx_0}{\theta(t)}, \qquad
\tb=\int_{t_0-\rho^2}^t\frac{\rmd s}{\theta^2(s)}=\frac{1}{a}\,
\ln\frac{\rho^2 }{t_0-t}\br. \label{2.5*}
\end{equation}
Then
\vspace{-8pt}
\begin{equation}
t\ =\ t_0-\rho^2\ \rme^{-a\tb} \qquad \mbox{and} \qquad \theta(t)\
=\ \sqrt{a}\ \rho\ \rme^{-\frac{1}{2}a\tb}. \label{2.6*}
\end{equation}

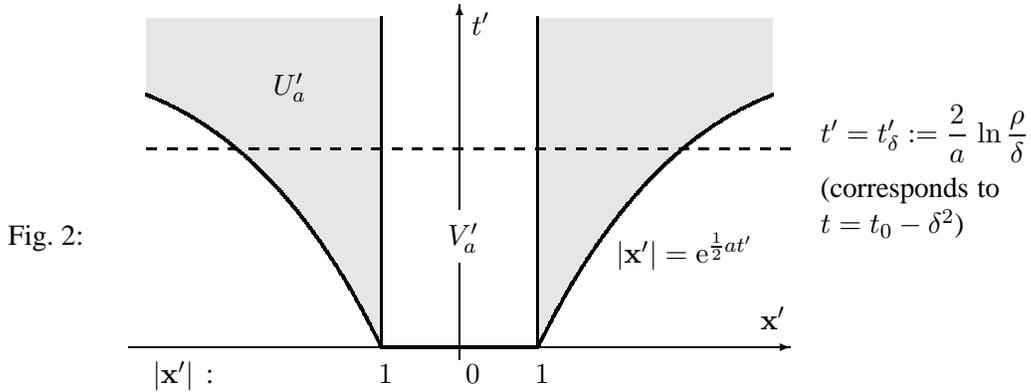
\begin{center}
\setlength{\unitlength}{0.8mm}
\begin{picture}(115,57)
  \color{lightgrey}
  % prava cast
  \linethickness{1mm} \put(58,52){\line(1,0){39}}
  \put(58,51.5){\line(1,0){39}} \put(58,51){\line(1,0){39}}
  \put(58,50.5){\line(1,0){39}} \put(58,50){\line(1,0){39}}
  \put(58,49.5){\line(1,0){39}} \put(58,49){\line(1,0){39}}
  \put(58,48.5){\line(1,0){39}} \put(58,48){\line(1,0){39}}
  \put(58,47.5){\line(1,0){39}} \put(58,47){\line(1,0){39}}
  \put(58,46.5){\line(1,0){39}} \put(58,46){\line(1,0){39}}
  \put(58,45.5){\line(1,0){39}} \put(58,45){\line(1,0){39}}
  \put(58,44.5){\line(1,0){39}} \put(58,44){\line(1,0){39}}
  \put(58,43.5){\line(1,0){39}} \put(58,43){\line(1,0){39}}
  \put(58,42.5){\line(1,0){39}} \put(58,42){\line(1,0){39}}
  \put(58,41.5){\line(1,0){39}} \put(58,41){\line(1,0){39}}
  \put(58,40.5){\line(1,0){39}} \put(58,40){\line(1,0){39}}
  \put(58,39.5){\line(1,0){38.2}} \put(58,39){\line(1,0){37}}
  \put(58,38.5){\line(1,0){36}} \put(58,38){\line(1,0){35}}
  \put(58,37.5){\line(1,0){34}} \put(58,37){\line(1,0){33}}
  \put(58,36.5){\line(1,0){32.1}} \put(58,36){\line(1,0){31.2}}
  \put(58,35.5){\line(1,0){30.4}} \put(58,35){\line(1,0){29.6}}
  \put(58,34.5){\line(1,0){28.9}} \put(58,34){\line(1,0){28.2}}
  \put(58,33.5){\line(1,0){27.5}} \put(58,33){\line(1,0){26.8}}
  \put(58,32.5){\line(1,0){26.1}} \put(58,32){\line(1,0){25.5}}
  \put(58,31.5){\line(1,0){24.8}} \put(58,31){\line(1,0){24.2}}
  \put(58,30.5){\line(1,0){23.6}} \put(58,30){\line(1,0){23.0}}
  \put(58,29.5){\line(1,0){22.5}} \put(58,29){\line(1,0){22.0}}
  \put(58,28.5){\line(1,0){21.45}} \put(58,28){\line(1,0){20.85}}
  \put(58,27.5){\line(1,0){20.35}} \put(58,27){\line(1,0){19.8}}
  \put(58,26.5){\line(1,0){19.3}} \put(58,26){\line(1,0){18.85}}
  \put(58,25.5){\line(1,0){18.4}} \put(58,25){\line(1,0){17.9}}
  \put(58,24.5){\line(1,0){17.45}} \put(58,24){\line(1,0){17.0}}
  \put(58,23.5){\line(1,0){16.55}} \put(58,23){\line(1,0){16.1}}
  \put(58,22.5){\line(1,0){15.65}} \put(58,22){\line(1,0){15.25}}
  \put(58,21.5){\line(1,0){14.85}} \put(58,21){\line(1,0){14.45}}
  \put(58,20.5){\line(1,0){14.05}} \put(58,20){\line(1,0){13.65}}
  \put(58,19.5){\line(1,0){13.25}} \put(58,19){\line(1,0){12.85}}
  \put(58,18.5){\line(1,0){12.45}} \put(58,18){\line(1,0){12.05}}
  \put(58,17.5){\line(1,0){11.7}} \put(58,17){\line(1,0){11.3}}
  \put(58,16.5){\line(1,0){10.93}} \put(58,16){\line(1,0){10.57}}
  \put(58,15.5){\line(1,0){10.23}} \put(58,15){\line(1,0){9.9}}
  \put(58,14.5){\line(1,0){9.6}} \put(58,14){\line(1,0){9.23}}
  \put(58,13.5){\line(1,0){8.91}} \put(58,13){\line(1,0){8.6}}
  \put(58,12.5){\line(1,0){8.28}} \put(58,12){\line(1,0){7.93}}
  \put(58,11.5){\line(1,0){7.60}} \put(58,11){\line(1,0){7.27}}
  \put(58,10.5){\line(1,0){6.95}} \put(58,10){\line(1,0){6.61}}
  \put(58,9.5){\line(1,0){6.30}} \put(58,9){\line(1,0){6.01}}
  \put(58,8.5){\line(1,0){5.70}} \put(58,8){\line(1,0){5.40}}
  \put(58,7.5){\line(1,0){5.12}} \put(58,7){\line(1,0){4.85}}
  \put(58,6.5){\line(1,0){4.51}} \put(58,6){\line(1,0){4.25}}
  \put(58,5.5){\line(1,0){4.00}} \put(58,5){\line(1,0){3.75}}
  \put(58,4.5){\line(1,0){3.45}} \put(58,4){\line(1,0){3.20}}
  \put(58,3.5){\line(1,0){2.91}} \put(58,3){\line(1,0){2.65}}
  \linethickness{0.5mm} \put(58.12,10){\line(0,-1){11.8}}
  \put(58.37,10){\line(0,-1){11.4}}
  \put(58.62,10){\line(0,-1){11.0}}
  \put(58.87,10){\line(0,-1){10.5}}
  \put(59.12,10){\line(0,-1){10.0}}
  \put(59.37,10){\line(0,-1){9.50}}
  \put(59.62,10){\line(0,-1){9.00}}
  \put(59.87,10){\line(0,-1){8.50}}
  \put(60.12,10){\line(0,-1){8.00}}
  \put(60.37,10){\line(0,-1){7.50}}
  % leva cast
  \linethickness{1mm} \put(32,52){\line(-1,0){39}}
  \put(32,51.5){\line(-1,0){39}} \put(32,51){\line(-1,0){39}}
  \put(32,50.5){\line(-1,0){39}} \put(32,50){\line(-1,0){39}}
  \put(32,49.5){\line(-1,0){39}} \put(32,49){\line(-1,0){39}}
  \put(32,48.5){\line(-1,0){39}} \put(32,48){\line(-1,0){39}}
  \put(32,47.5){\line(-1,0){39}} \put(32,47){\line(-1,0){39}}
  \put(32,46.5){\line(-1,0){39}} \put(32,46){\line(-1,0){39}}
  \put(32,45.5){\line(-1,0){39}} \put(32,45){\line(-1,0){39}}
  \put(32,44.5){\line(-1,0){39}} \put(32,44){\line(-1,0){39}}
  \put(32,43.5){\line(-1,0){39}} \put(32,43){\line(-1,0){39}}
  \put(32,42.5){\line(-1,0){39}} \put(32,42){\line(-1,0){39}}
  \put(32,41.5){\line(-1,0){39}} \put(32,41){\line(-1,0){39}}
  \put(32,40.5){\line(-1,0){39}} \put(32,40){\line(-1,0){39}}
  \put(32,39.5){\line(-1,0){38.2}} \put(32,39){\line(-1,0){37}}
  \put(32,38.5){\line(-1,0){36}} \put(32,38){\line(-1,0){35}}
  \put(32,37.5){\line(-1,0){34}} \put(32,37){\line(-1,0){33}}
  \put(32,36.5){\line(-1,0){32.1}} \put(32,36){\line(-1,0){31.2}}
  \put(32,35.5){\line(-1,0){30.4}} \put(32,35){\line(-1,0){29.6}}
  \put(32,34.5){\line(-1,0){28.9}} \put(32,34){\line(-1,0){28.2}}
  \put(32,33.5){\line(-1,0){27.5}} \put(32,33){\line(-1,0){26.8}}
  \put(32,32.5){\line(-1,0){26.1}} \put(32,32){\line(-1,0){25.5}}
  \put(32,31.5){\line(-1,0){24.8}} \put(32,31){\line(-1,0){24.2}}
  \put(32,30.5){\line(-1,0){23.6}} \put(32,30){\line(-1,0){23.0}}
  \put(32,29.5){\line(-1,0){22.5}} \put(32,29){\line(-1,0){22.0}}
  \put(32,28.5){\line(-1,0){21.45}} \put(32,28){\line(-1,0){20.85}}
  \put(32,27.5){\line(-1,0){20.35}} \put(32,27){\line(-1,0){19.8}}
  \put(32,26.5){\line(-1,0){19.3}} \put(32,26){\line(-1,0){18.85}}
  \put(32,25.5){\line(-1,0){18.4}} \put(32,25){\line(-1,0){17.9}}
  \put(32,24.5){\line(-1,0){17.45}} \put(32,24){\line(-1,0){17.0}}
  \put(32,23.5){\line(-1,0){16.55}} \put(32,23){\line(-1,0){16.1}}
  \put(32,22.5){\line(-1,0){15.65}} \put(32,22){\line(-1,0){15.25}}
  \put(32,21.5){\line(-1,0){14.85}} \put(32,21){\line(-1,0){14.45}}
  \put(32,20.5){\line(-1,0){14.05}} \put(32,20){\line(-1,0){13.65}}
  \put(32,19.5){\line(-1,0){13.25}} \put(32,19){\line(-1,0){12.85}}
  \put(32,18.5){\line(-1,0){12.45}} \put(32,18){\line(-1,0){12.05}}
  \put(32,17.5){\line(-1,0){11.7}} \put(32,17){\line(-1,0){11.3}}
  \put(32,16.5){\line(-1,0){10.93}} \put(32,16){\line(-1,0){10.57}}
  \put(32,15.5){\line(-1,0){10.23}} \put(32,15){\line(-1,0){9.9}}
  \put(32,14.5){\line(-1,0){9.6}} \put(32,14){\line(-1,0){9.23}}
  \put(32,13.5){\line(-1,0){8.91}} \put(32,13){\line(-1,0){8.6}}
  \put(32,12.5){\line(-1,0){8.28}} \put(32,12){\line(-1,0){7.93}}
  \put(32,11.5){\line(-1,0){7.60}} \put(32,11){\line(-1,0){7.27}}
  \put(32,10.5){\line(-1,0){6.95}} \put(32,10){\line(-1,0){6.61}}
  \put(32,9.5){\line(-1,0){6.30}} \put(32,9){\line(-1,0){6.01}}
  \put(32,8.5){\line(-1,0){5.70}} \put(32,8){\line(-1,0){5.40}}
  \put(32,7.5){\line(-1,0){5.12}} \put(32,7){\line(-1,0){4.85}}
  \put(32,6.5){\line(-1,0){4.51}} \put(32,6){\line(-1,0){4.25}}
  \put(32,5.5){\line(-1,0){4.00}} \put(32,5){\line(-1,0){3.75}}
  \put(32,4.5){\line(-1,0){3.45}} \put(32,4){\line(-1,0){3.20}}
  \put(32,3.5){\line(-1,0){2.91}} \put(32,3){\line(-1,0){2.65}}
  \linethickness{0.5mm} \put(31.88,10){\line(0,-1){11.8}}
  \put(31.63,10){\line(0,-1){11.4}}
  \put(31.38,10){\line(0,-1){11.0}}
  \put(31.13,10){\line(0,-1){10.5}}
  \put(30.88,10){\line(0,-1){10.0}}
  \put(30.63,10){\line(0,-1){9.50}}
  \put(30.38,10){\line(0,-1){9.00}}
  \put(30.13,10){\line(0,-1){8.50}}
  \put(29.87,10){\line(0,-1){8.00}}
  \put(29.63,10){\line(0,-1){7.50}}
  \color{white}
  \qbezier(58,-2.45)(74,30.55)(97,39.55)
  \qbezier(31.5,-2.45)(15.5,30.55)(-7.5,39.55)
  \qbezier(31.6,-2)(16,30.55)(-7,39.7)
  \qbezier(31.1,-2)(15.5,30.55)(-7.5,39.7)
  \color{black}
  \thicklines \put(32,-2){\line(0,1){55}}
  \put(58,-2){\line(0,1){55}} \qbezier(32,-2)(16,31)(-7,40)
  \qbezier(58,-2)(74,31)(97,40) \put(32,-2){\line(1,0){26}}
  \thinlines
  \put(-10,-2){\vector(1,0){110}} \put(95,1){$\bfxb$}
  \put(45,-4){\line(0,1){16}}
  \put(45,21){\vector(0,1){34}} \put(47,50){$\tb$}
  \dashline[+20]{1.8}(-7,31)(100,31)
  %\dashline[+20]{1.8}(25,-2)(25,53)
  %\dashline[+20]{1.8}(65,-2)(65,53)
  \put(-6,-8){$|\bfxb|$\, :}
  \put(31.5,-8){$1$} \put(57.5,-8){$1$}
  \put(46,-8){$0$}
  %\put(12,0.5){$1+\xi$} \put(67,0.5){$1+\xi$}
  \put(105,32){$\displaystyle \tb=\tb_{\delta}:=\frac{2}{a}\,
  \ln\frac{\rho}
  {\delta}$}
  \put(105,23){(corresponds to}
  \put(105,17){$t=t_0-\delta^2$)}
  \put(43,15){$\Vab$} \put(14,40){$\Uab$}
  \put(71,12){$|\bfxb|=\rme^{\frac{1}{2}a\tb}$}
  \put(-30,15){Fig.~2:}
  \end{picture}
\end{center}

\vspace{16pt} \noindent
The time interval $(t_0-\rho^2,t_0)$ on the $t$--axis now
corresponds to the interval $\bigl(0,\, \infty\bigr)$ on the
$\tb$--axis. Equations (\ref{2.5*}) represent a one--to--one
transformation of the parabolic region $\Var$ in the
$\bfx,t$--space onto the infinite stripe
\begin{displaymath}
\Vab\ :=\ \bigl\{(\bfxb,\tb)\in\R^4;\ \tb>0\ \mbox{and}\
|\bfxb|<1\bigr\}
\end{displaymath}
in the $\bfxb,\tb$--space. Similarly, (\ref{2.5*}) is a
one--to--one transformation of set $\Uar$ in the $\bfx,t$--space
onto
\vspace{-10pt}
\begin{displaymath}
\Uab\ :=\ \bigl\{(\bfxb,\tb)\in\R^4;\ \tb>0\ \mbox{and}\
1<|\bfxb|<\rme^{\frac{1}{2}a\tb}\bigr\}
\end{displaymath}
in the $\bfxb,\tb$--space. We denote
\vspace{-8pt}
\begin{equation}
\tb_{\delta}\ :=\ \frac{2}{a}\, \ln\br\frac{\rho}{\delta}.
\label{2.6a*}
\end{equation}
Then $\tb=\tb_{\delta}$ corresponds to $t=t_0-\delta^2$.
Obviously, \ $\rme^{-\frac{1}{2}a\tb_{\delta}}=\delta/\rho$ and
$\delta\to 0+$ corresponds to $\tb_{\delta}\to\infty$. If we put
\vspace{-4pt}
\begin{eqnarray}
\bfv(\bfx,t) &=& \frac{1}{\theta(t)}\
\bfvb\Bigl(\frac{\bfx-\bfx_0}{\theta(t)},\ \frac{1}{a}\
\ln\frac{\rho^2}{t_0-t}\Bigr), \label{2.7*} \\ [2pt]
p(\bfx,t) &=& \frac{1}{\theta^2(t)}\
\pb\Bigl(\frac{\bfx-\bfx_0}{\theta(t)},\ \frac{1}{a}\
\ln\frac{\rho^2}{t_0-t}\Bigr) \label{2.8*}
\end{eqnarray}
then functions $\bfvb$, $\pb$ represent a suitable weak solution
of the system of equations
\begin{eqnarray}
\partial_{\tb}\bfvb+\bfvb\cdot\nablab\bfvb &=& -\nablab\pb+
\nu\Deltab\bfvb-{\textstyle\frac{1}{2}}\br a\br\bfvb-
{\textstyle\frac{1}{2}}\br a\bfxb\cdot\nablab\bfvb, \label{2.9*} \\
[2pt]
\divb\bfvb &=& 0 \label{2.10*}
\end{eqnarray}
in any bounded sub--domain of $\Qab:=\bigl\{(\bfxb,\tb)\in\R^4;\
\tb>0\ \mbox{and}\ |\bfxb|<\rme^{\frac{1}{2}a\tb}\bigr\}$. (The
symbols $\nablab$ and $\Deltab$ denote the nabla operator and the
Laplace operator with respect to the spatial variable $\bfxb$.)

\vspace{4pt} \noindent
{\bf Sets $\M{R_1}{R_2}(t)$, $\Mb{R_1}{R_2}$ and $\Bb{R_1}$.} \
Let $0<R_1<R_2$. We denote $\M{R_1}{R_2}(t):=\{\bfx\in\R^3;\
R_1\br\theta(t)<|\bfx-\bfx_0|<R_2\br\theta(t)\}$ and
$\Mb{R_1}{R_2}:=\{\bfxb\in\R^3;\ R_1<|\bfxb|<R_2\}$. In order to
keep a consistent notation, we also denote by $\Bb{R_1}$ the ball
$\{\bfxb\in\R^3;\ |\bfxb|<R_1\}$. The mapping
$\bfx\mapsto\bfx'=(\bfx-\bfx_0)/ \theta(t)$ is a one--to--one
transformation of $\M{R_1}{R_2}(t)$ onto $\Mb{R_1}{R_2}$ and
$B_{R_1\theta(t)}(\bfx_0)$ onto $\Bb{R_1}$ at each time instant
$t\in(t_0-\rho^2,t_0)$.

\vspace{4pt} \noindent
{\bf The first estimate of $G^{I\hspace{-0.3pt}I}(\delta)$.} \
Recall that $\tb=\tb_{\delta}=2a^{-1}\, \ln(\rho/\delta)$
corresponds to $t=t_0-\delta^2$ (see formulas (\ref{2.6*}) and
(\ref{2.6a*})). Suppose that $\varphi$ is an infinitely
differentiable function in $\R^3$ such that
\begin{equation}
\varphi(\bfxb)\ \ \left\{ \begin{array}{ll} =1 & \mbox{for}\
|\bfxb|\leq 3, \\ [4pt] \in[0,1] & \mbox{for}\ 3<|\bfxb|\leq 4, \\
[4pt] =0 & \mbox{for}\ |\bfxb|>4. \end{array} \right.
\label{2.12*}
\end{equation}
Transforming $G^{I\hspace{-0.3pt}I}(\delta)$ to the variables
$\bfxb$, $\tb$, we get
\begin{align}
G^{I\hspace{-0.3pt}I}& (\delta)\, =\,
\frac{a\br\rho^2}{\delta^2}\int_{\tb_{\delta}}^{\infty}
\int_{\Bb{1}} |\bfv'|^3\; \rmd\bfx'\ \rme^{-a\tb}\, \rmd\br\tb\leq
\frac{a\br\rho^2}{\delta^2}\int_{\tb_{\delta}}^{\infty}
\|\bfvb\|_{6;\, \Bb{1}}^{\frac{3}{2}}\, \|\bfvb\|_{2;\,
\Bb{1}}^{\frac{3}{2}}\, \rme^{-a\tb}\; \rmd\tb \nonumber \\
\noalign{\vskip 2pt}
& \leq\ \frac{a\br\rho^2}{\delta^2}\int_{\tb_{\delta}}^{\infty}
\|\varphi\br\bfvb\|_{6;\, \Bb{4}}^{\frac{3}{2}}\,
\|\varphi\br\bfvb\|_{2;\,
\Bb{4}}^{\frac{3}{2}}\, \rme^{-a\tb}\; \rmd\tb \nonumber \\
\noalign{\vskip 2pt}
& \leq\ \frac{2}{3^{\frac{3}{4}}\br\pi}\,
\frac{a\br\rho^2}{\delta^2} \int_{\tb_{\delta}}^{\infty}
\|\nablab(\varphi\br\bfvb)\|_{2;\, \Bb{4}}^{\frac{3}{2}}\,
\|\varphi\br\bfvb\|_{2;\,
\Bb{4}}^{\frac{3}{2}}\, \rme^{-a\tb}\; \rmd\tb \nonumber \\
\noalign{\vskip 2pt}
& \leq\ \frac{2}{3^{\frac{3}{4}}\br\pi}\,
\frac{a\br\rho^2}{\delta^2}\,
\biggl(\int_{\tb_{\delta}}^{\infty}\|\nablab(\varphi\br\bfvb)
\|_{2;\, \Bb{4}}^2\, \rme^{-\frac{2}{3}a\tb}\, \rmd\tb\biggr)^{\!
\frac{3}{4}}\,
\biggl(\int_{\tb_{\delta}}^{\infty}\|\varphi\br\bfvb\|_{2;\,
\Bb{4}}^6\, \rme^{-2a\tb}\, \rmd\tb\biggr)^{\! \frac{1}{4}}
\nonumber \\
& =\ \frac{2}{3^{\frac{3}{4}}\br\pi}\, a\,
\biggl(\int_{\tb_{\delta}}^{\infty}\|\nablab
(\varphi\br\bfvb)\|_{2;\, \Bb{4}}^2\,
\rme^{-\frac{2}{3}a(\tb-\tb_{\delta})}\, \rmd\tb\biggr)^{\!
\frac{3}{4}} \nonumber \\ \noalign{\vskip-4pt}
& \hspace{100pt} \cdot
\biggl(\int_{\tb_{\delta}}^{\infty}\|\varphi\br\bfvb\|_{2;\,
\Bb{4}}^6\, \rme^{-2a(\tb-\tb_{\delta})}\, \rmd\tb\biggr)^{\!
\frac{1}{4}}. \label{2.13*}
\end{align}
The factor $2/(3^{\frac{3}{4}}\br\pi)$ comes from Sobolev's
inequality, see e.g.~\cite[p.~34]{Ta}. In order to estimate the
integrals on the right hand side of (\ref{2.13*}), we use the next
lemma and the generalized energy inequality in the
$\bfxb,\tb$--space.

\vspace{4pt} \noindent
\begin{lemma}
\label{L4}
Assume that $0<\alpha\leq r$, \ $0<\beta\leq s$, \ $R>1$, \
$\tb_{\delta}>2a^{-1}\, \ln\br R$, \ and at least one of the two
conditions

\vspace{10pt}
(a) \quad $\alpha=r$, \ $\omega\geq 0$, \qquad (b) \quad
$\alpha<r$, \ $\omega>0$

\vspace{10pt} \noindent
holds. Then
\begin{equation}
\int_{\tb_{\delta}}^{\infty} \biggl(
\int_{\Mb{1}{R}}|\bfvb|^{\beta}\; \rmd\bfxb\biggr)^{\!
\frac{\alpha}{\beta}}\; \rme^{-\omega a(\tb-\tb_{\delta})}\;
\rmd\br\tb\ \longrightarrow\ 0 \qquad \mbox{as}\ \
\tb_{\delta}\to\infty. \label{2.11*}
\end{equation}
\end{lemma}

\begin{proof}
We use $C$ as a generic constant independent of $\delta$. In order
to indicate that $C$ may depend on other quantities, we often
write e.g.~$C(R,a)$, $C(R,\beta)$ or similar. We have
\begin{align}
\int_{\tb_{\delta}}^{\infty} & \biggl(
\int_{\Mb{1}{R}}|\bfvb|^{\beta}\; \rmd\bfxb\biggr)^{\!
\frac{\alpha}{\beta}}\; \rme^{-\omega a(\tb-\tb_{\delta})}\;
\rmd\br\tb\ \leq\ C(R,\beta)\int_{\tb_{\delta}}^{\infty}\biggl(
\int_{\Mb{1}{R}}|\bfvb|^s\; \rmd\bfxb\biggr)^{\!
\frac{\alpha}{s}}\;
\rme^{-\omega a(\tb-\tb_{\delta})}\; \rmd\br\tb \nonumber \\
\noalign{\vskip 1pt}
&=\ \frac{C(R,\beta,\rho)}{\delta^{2\omega}}
\int_{t_0-\delta^2}^{t_0} \biggl(\int_{\M{1}{R}(t)} |\bfv|^s\;
\rmd\bfx\biggr)^{\! \frac{\alpha}{s}}\,
\theta^{\br2\omega+\alpha-3\frac{\alpha}{s}-2}(t)\; \rmd\br t.
\label{2.18*}
\end{align}
If condition (a) holds then the exponent
$2\omega+\alpha-3\alpha/s-2$ equals $2\omega+r\, (
1-3/s-2/r)=2\omega$. Hence the right hand side of (\ref{2.18*}) is
less than or equal to
\begin{displaymath}
C(R,\beta,\rho)\int_{t_0-\delta^2}^{t_0} \biggl(\int_{\M{1}{R}(t)}
|\bfv|^s\; \rmd\bfx\biggr)^{\! \frac{r}{s}}\; \rmd\br t.
\end{displaymath}
This tends to zero as $\delta\to 0+$ due to (\ref{1.6*}). If
condition (b) holds then the right hand side of (\ref{2.18*}) is
less than or equal to
\begin{displaymath}
\frac{C(R,\beta,\rho)}{\delta^{2\omega}}\, \biggl[
\int_{t_0-\delta^2}^{t_0} \biggl(\int_{\M{1}{R}(t)} |\bfv|^s\;
\rmd\bfx\biggr)^{\! \frac{r}{s}}\, \rmd\br
t\biggr]^{\frac{\alpha}{r}}\, \biggl[\int_{t_0-\delta^2}^{t_0}
\theta^{\left[2\omega+\alpha-3\frac{\alpha}{s}-2\right]\,
\frac{r}{r-\alpha}}(t)\; \rmd\br t\biggr]^{\frac{r-\alpha}{r}}.
\end{displaymath}
The last factor on the right hand side is
\begin{displaymath}
\biggl[\int_{t_0-\delta^2}^{t_0}
\theta^{\left[2\omega+\alpha-3\frac{\alpha}{s}-2\right]\,
\frac{r}{r-\alpha}}(t)\; \rmd\br t\biggr]^{\frac{r-\alpha}{r}}\ =\
\biggl[\int_{t_0-\delta^2}^{t_0} \theta^{-2+\frac{2\omega
r}{r-\alpha}}(t)\; \rmd\br t\biggr]^{\frac{r-\alpha}{r}}\ =\
C(a,\omega)\ \delta^{2\omega}.
\end{displaymath}
This shows that the right hand side of (\ref{2.18*}) tends to zero
for $\delta\to 0+$ in the case of condition (b) as well. The proof
is completed.
\end{proof}

\vspace{4pt} \noindent
{\bf The generalized energy inequality in the $\bfxb$,
$\tb$--space.} \ Since $\bfvb$, $\pb$ is a suitable weak solution
to the system (\ref{2.9*}), (\ref{2.10*}), it satisfies (by
analogy with (\ref{1.3*})) the generalized energy inequality
\begin{align}
2\nu\int_{\Qab}|\nablab\bfvb|^2\, \phi\; \rmd\bfxb\, \rmd\br\tb\
\leq\ \int_{\Qab}\bigl[\, & |\bfvb|^2\,
\bigl(\partial_{\tb}\phi+\nu\Deltab\phi\bigr)+\bigl(|\bfvb|^2+
2\pb\bigr)\, \bfvb\cdot\nablab\phi \nonumber \\
& +{\textstyle\frac{1}{2}}\br a\, |\bfvb|^2\,
\phi+{\textstyle\frac{1}{2}}\br a\, (\bfxb\cdot\nablab\phi)\,
|\bfvb|^2 \br\bigr]\; \rmd\bfxb\, \rmd\br\tb \label{2.14*}
\end{align}
for every non--negative function $\phi$ from $C^{\infty}_0(\Qab)$.

Due to technical reasons, we further assume that
$0<\delta<\rho/4$. This assumption implies that
$\tb_{\delta}>2a^{-1}\, \ln\br 4$.

If function $\phi$ in inequality (\ref{2.14*}) has the form
$\phi(\bfxb,\tb)=\varphi(\bfxb)\, \vartheta(\tb)$, where $\varphi$
is defined in (\ref{2.12*}) and $\vartheta$ is a
$C^{\infty}$--function in $\bigl(2a^{-1}\br\ln\br 4,\infty\bigr)$
with a compact support, we get
\begin{align}
& 2\nu\int_0^{\infty}\int_{\Bb{4}}|\nablab\bfvb|^2\,
\varphi^2(\bfxb)\, \vartheta(\tb)\; \rmd\bfxb\, \rmd\br\tb
\nonumber \\ \noalign{\vskip 2pt}
& \leq\ \int_0^{\infty}\int_{\Bb{4}}\bigl[\,
|\varphi(\bfxb)\br\bfvb|^2\, \dot{\vartheta}(\tb)+\nu\,
|\bfvb|^2\, \Delta'\varphi^2(\bfxb)\, \vartheta(\tb)
+\bigl(|\bfvb|^2+2\pb\bigr)\, \bfvb\cdot\nablab\varphi^2(\bfxb)\,
\vartheta(\tb) \nonumber \\
& \hspace{70pt} +{\textstyle\frac{1}{2}}\br a\,
|\varphi(\bfxb)\br\bfvb|^2\, \vartheta(\tb)+
{\textstyle\frac{1}{2}}\br a\, (\bfxb\cdot\nablab
\varphi^2(\bfxb))\, |\bfvb|^2\, \vartheta(\tb) \br\bigr]\;
\rmd\bfxb\, \rmd\br\tb. \label{2.15*}
\end{align}
Choosing $\vartheta(\tb)=\rme^{-\frac{2}{3}a(\tb-\tb_{\delta})}\,
[\cR_{1/m}\chi](\tb)$, where $\chi$ is the characteristic function
of the interval $(\tb_{\delta},\tb)$ and $\cR_{1/m}$ is a
one--dimensional mollifier with the kernel supported in
$(-1/m,1/m)$, and letting $m\to\infty$, we obtain
\begin{align}
\|\varphi\br & \bfvb(\, .\, ,\tb)\|_{2;\, \Bb{4}}^2\,
\rme^{-\frac{2}{3}a(\tb-\tb_{\delta})}+\frac{a}{6}
\int_{\tb_{\delta}}^{\tb}\|\varphi\br\bfvb(\, .\, ,\tau)\|_{2;\,
\Bb{4}}^2\, \rme^{-\frac{2}{3}a(\tau-\tb_{\delta})}\;
\rmd\tau \nonumber \\
& \hspace{17pt}
+2\nu\int_{\tb_{\delta}}^{\tb}\|\varphi\br\nablab\bfvb(\, .\,
,\tau)\|_{2;\, \Bb{4}}^2\,
\rme^{-\frac{2}{3}a(\tau-\tb_{\delta})}\; \rmd\tau \nonumber \\
& \leq\ \|\varphi\br\bfvb(\, .\, ,\tb_{\delta})\|_{2;\, \Bb{4}}^2
+\int_{\tb_{\delta}}^{\tb}\int_{\Bb{4}} \bigl[
-\nu\nablab|\bfvb|^2\cdot \nablab\varphi^2+
\bigl(|\bfvb|^2+2\pb\bigr)\, (\bfvb\cdot\nablab\varphi^2)
\nonumber \\
& \hspace{144pt} + \bigl({\textstyle\frac{1}{2}}\br
a\bfxb\cdot\nablab\varphi^2\bigr)\, |\bfvb|^2 \bigr]\; \rmd\bfxb\;
\rme^{-\frac{2}{3}a(\tau-\tb_{\delta})}\; \rmd\tau. \label{2.16*}
\end{align}
This inequality holds for a.a.~$\tb_{\delta}>2a^{-1}\, \ln\br 4$
and all $\tb\geq\tb_{\delta}$. Note that it can also be formally
obtained, multiplying (\ref{2.9*}) by $2\bfvb\, \varphi^2\,
\rme^{-\frac{2}{3}a (\tb-\tb_{\delta})}$ and integrating in
$\Bb{4}\times (\tb_{\delta},\tb)$. The second term on the left
hand side of inequality (\ref{2.16*}) comes from the integral of
$|\varphi(\bfxb)\br\bfvb|^2\, [\dot{\vartheta}(\tb)+
\frac{1}{2}a\br\vartheta(\tb)]$ on the right hand side of
(\ref{2.15*}). Using the formula $\varphi^2\, |\nablab\bfvb|^2=
|\nablab(\varphi\bfvb)|^2- |\nablab\varphi^2|\, |\bfvb|^2
-\frac{1} {2}\, \nablab\varphi^2\cdot\nablab|\bfvb|^2$, we can
further rewrite (\ref{2.16*}) as follows:
\begin{align}
\|\varphi\br & \bfvb(\, .\, ,\tb)\|_{2;\, \Bb{4}}^2\,
\rme^{-\frac{2}{3}a(\tb-\tb_{\delta})}+\frac{a}{6}
\int_{\tb_{\delta}}^{\tb}\|\varphi\br\bfvb(\, .\, ,\tau)\|_{2;\,
\Bb{4}}^2\, \rme^{-\frac{2}{3}a(\tau-\tb_{\delta})}\;
\rmd\tau \nonumber \\
&
\hspace{17pt}+2\nu\int_{\tb_{\delta}}^{\tb}\|\nablab(\varphi
\bfvb(\, .\, ,\tau))\|_{2;\, \Bb{4}}^2\,
\rme^{-\frac{2}{3}a(\tau-\tb_{\delta})}\; \rmd\tau \nonumber \\
& \leq\ \|\varphi\br\bfvb(\, .\, ,\tb_{\delta})\|_{2;\,
\Bb{4}}^2+\int_{\tb_{\delta}}^{\tb} \int_{\Bb{4}}\bigl[2\nu\,
|\nablab\varphi|^2\, |\bfvb|^2+\bigl(|\bfvb|^2+2\pb\bigr)\,
(\bfvb\cdot\nablab\varphi^2) \nonumber \\
& \hspace{144pt} +\bigl({\textstyle\frac{1}{2}}\br
a\bfxb\cdot\nablab\varphi^2\bigr)\, |\bfvb|^2\bigr]\; \rmd\bfxb\;
\rme^{-\frac{2}{3}a(\tau-\tb_{\delta})}\; \rmd\tau.
\label{2.17*}
\end{align}

%-----------------------------------------------------------------------
\section{Estimates of the right hand side of inequality (\ref{2.17*})}
\label{S3}

The right hand side of inequality (\ref{2.17*}) can be estimated
from above by the sum of \\ $\|\varphi\bfvb(\, .\,
,\tb_{\delta})\|_{2;\, \Bb{4}}^2$ and the two terms $K^I(\delta)$,
$K^{I\hspace{-0.3pt}I}(\delta)$, where
\newpage
\begin{alignat*}{3}
& K^I(\delta) &&:=\ \int_{\tb_{\delta}}^{\infty}
\int_{\Mb{3}{4}}\bigl( 2\nu\, |\nablab\varphi|^2\,
|\bfvb|^2+|\bfvb|^2\, |\bfvb\cdot\nablab\varphi^2| \\
\noalign{\vskip-1pt}
& && \hspace{60pt} + \bigl|{\textstyle\frac{1}{2}}\br
a\bfxb\cdot\nablab\varphi^2|\ |\bfvb|^2\bigr)\; \rmd\bfxb\,
\rme^{-\frac{2}{3}\br a(\tau-\tb_{\delta})}\, \rmd\tau,
\\ \noalign{\vskip 4pt}
& K^{I\hspace{-0.3pt}I}(\delta)\ &&:=\
\int_{\tb_{\delta}}^{\infty}\int_{\Mb{3}{4}} \bigl|\br 2\pb\,
(\bfvb\cdot\nablab\varphi^2)\bigr|\; \rmd\bfxb\;
\rme^{-\frac{2}{3}a(\tau-\tb_{\delta})}\; \rmd\tau.
\end{alignat*}
Due to Lemma \ref{L4}, $K^I(\delta)\to 0$ for $\delta\to 0+$.
$K^{I\hspace{-0.3pt}I}(\delta)$ can be estimated as follows:
\phantom{$\cn21$}
\begin{align}
& K^{I\hspace{-0.3pt}I}(\delta)\ \leq\
C\int_{\tb_{\delta}}^{\infty}\int_{\Mb{3}{4}}|\pb|\, |\bfvb|\;
\rmd\bfxb\; \rme^{-\frac{2}{3}a(\tau-\tb_{\delta})}\; \rmd\tau
\nonumber \\ \noalign{\vskip 1pt}
& \leq\ C\, \biggl[\br\int_{\tb_{\delta}}^{\infty}\!\biggl(
\int_{\Mb{3}{4}}\!\!|\bfvb|^{\frac{s}{s-2}}\, \rmd\bfxb\biggr)^{\!
\frac{r(s-2)}{s}}\rmd\tau\br \biggr]^{\frac{1}{r}}\
\biggl[\br\int_{\tb_{\delta}}^{\infty}\!
\biggl(\int_{\Mb{3}{4}}\!\!|\pb|^{\frac{s}{2}}\,
\rmd\bfxb\biggr)^{\! \frac{2}{s}\, \frac{r}{r-1}}\,
\rme^{-\frac{2}{3}\, \frac{r}{r-1}\, a(\tau-\tb_{\delta})}\;
\rmd\tau\br\biggr]^{\frac{r-1}{r}} \nonumber
\\ \noalign{\vskip 4pt}
& =\ \cc21(\delta)\ {\cP}^{\frac{r-1}{r}}(\delta),
\label{3.1*}
\end{align}
where
\vspace{-8pt}
\begin{align*}
\cc21(\delta)\ &:=\ C\ \biggl[\int_{\tb_{\delta}}^{\infty}
\biggl(\int_{\Mb{3}{4}}|\bfvb|^{\frac{s}{s-2}}\;
\rmd\bfxb\biggr)^{\! \frac{r(s-2)}{s}}\,
\rmd\tau\biggr]^{\frac{1}{r}}\ \longrightarrow\ 0 \quad
\mbox{for}\ \delta\to 0+ \quad \mbox{(due to Lemma \ref{L4})},
\\
\cP(\delta)\ &:=\ \int_{\tb_{\delta}}^{\infty}\biggl(
\int_{\Mb{3}{4}}|\pb|^{\frac{s}{2}}\,
\rmd\bfxb\biggr)^{\frac{2}{s}\, \frac{r}{r-1}}\,
\rme^{-\frac{2}{3}\, \frac{r}{r-1}\, a(\tau-\tb_{\delta})}\;
\rmd\tau.
\end{align*}
In order to estimate $\cP(\delta)$, we use the next lemma.

\begin{lemma}
\label{L2}
Let $\, 0<\gamma<1$. Then there exist constants $\cn31$, $\cn32$
and $\cn33$ so that the inequality
\begin{align}
\int_{\Mb{3}{4}}|\pb|^{\frac{s}{2}}\, \rmd\bfxb\ &\leq\ \cc31\,
\biggl(\int_{\Bb{1}}|\bfvb|^2\, \rmd\bfxb\biggr)^{\!
\frac{s}{2}}+\cc32\int_{\Mbu{$1$\hspace{2.7pt}}{$\rme^{a\tb
\hspace{-1.5pt}/2}$}} |\bfvb|^s\, \rmd\bfxb \nonumber \\
& \hspace{17pt} +\cc33\,
\biggl(\rme^{-\frac{3}{2}a\tb}\int_{\Mbuu{$\gamma\br\rme^{a\tb
\hspace{-1.5pt}/2}$}{$\rme^{a\tb \hspace{-1.5pt}/2}$}}
\bigl[\br|\bfvb|^2+|\pb|\br\bigr]\; \rmd\bfxb\biggr)^{\!
\frac{s}{2}} \label{3.2*}
\end{align}
holds for a.a.~$\tb>2a^{-1}\, \ln(4/\gamma)$.
\end{lemma}

\begin{proof}
Let $\eta$ be an infinitely differentiable cut--off function in
$\R^3$ such that
\begin{displaymath}
\eta(\bfxb)\ \ \left\{ \begin{array}{ll} =1 & \mbox{for}\
|\bfxb|\leq\gamma\, \rme^{\frac{1}{2}a\tb}, \\ [4pt] \in[0,1] &
\mbox{for}\
\gamma\, \rme^{\frac{1}{2}a\tb}<|\bfxb|\leq \rme^{\frac{1}{2}a\tb}, \\
[4pt] =0 & \mbox{for}\ \rme^{\frac{1}{2}a\tb}<|\bfxb|
\end{array} \right.
\end{displaymath}
and \hspace{60pt} $\displaystyle |\nablab\eta|\ \leq\
\frac{2}{1-\gamma}\, \rme^{-\frac{1}{2}a\tb}$ \qquad and \qquad
$\displaystyle |{\nablab}^2\eta|\ \leq\ \frac{8}{(1-\gamma)^2}\,
\rme^{-a\tb}$.

\vspace{6pt} \noindent
Function $\eta$ can be further expressed in the form
$\eta_1+\eta_2$, where both the functions $\eta_1$ and $\eta_2$
are from $C_0^{\infty}(\R^3)$, with values in $[0,1]$, and such
that $\eta_1=1$ on $\Bb{1}$ and $\eta_1=0$ on $\R^3\smallsetminus
\Bb{2}$. Thus, function $\eta_1$ is supported in the closure of
$\Bb{2}$ and $\eta_2$ is supported in $\R^3\smallsetminus\Bb{1}$.

The function $\eta\br\pb$ satisfies the obvious identity
\begin{displaymath}
\eta(\bfxb)\, \pb(\bfxb,\tb)\ =\
-\frac{1}{4\pi}\int_{\R^3}\frac{1}{|\bfxb-\bfyb|}\
\bigl[\Deltab(\eta\br\pb)\bigr](\bfyb,\tb)\; \rmd\bfyb
\end{displaymath}
for $\bfxb\in\R^3$. Integrating by parts and using the formula
$\Deltab\pb=-\partialb_i\partialb_j(\vb_i\vb_j)$ (which we obtain
if we apply operator ${\rm div}'$ to equation (\ref{2.9*})), we
derive that
\begin{equation}
\eta(\bfxb)\, \pb(\bfxb,\tb)\ =\ \pb_1(\bfxb,\tb)+\pb_2(\bfxb,\tb)
+\pb_3(\bfxb,\tb), \label{3.3*}
\end{equation}
where
\vspace{-8pt}
\begin{eqnarray*}
\pb_1(\bfxb,\tb) &=& -\frac{1}{4\pi} \int_{\Bb{2}}
\frac{\partial^2}{\partial\yb_i\,
\partial\yb_j}\Bigl( \frac{1}{|\bfxb-\bfyb|} \Bigr)\,
[\br\eta_1\br\vb_i\vb_j](\bfyb,\tb)\; \rmd\bfyb, \\ [2pt]
\pb_2(\bfxb,\tb) &=& -\frac{1}{4\pi}
\int_{\Mbuu{1}{$\rme^{a\tb\hspace{-1.5pt}/2}$}}\frac{\partial^2}
{\partial\yb_i\, \partial\yb_j}\Bigl( \frac{1}{|\bfxb-\bfyb|}
\Bigr)\, [\br\eta_2\br\vb_i\vb_j](\bfyb,\tb)\; \rmd\bfyb, \\
[2pt]
\pb_3(\bfxb,\tb) &=& \frac{1}{2\pi}
\int_{\Mbuu{$\gamma\br\rme^{a\tb\hspace{-1.5pt}/2}$}{$\rme^{a\tb
\hspace{-1.5pt}/2}$}} \frac{\xb_i-\yb_i}{|\bfxb-\bfyb|^3}\, \Bigl(
\frac{\partial\eta}{\partial\yb_j}\,
\vb_i\br\vb_j\Bigr)(\bfyb,\tb)\; \rmd\bfyb \\ [2pt]
&& +\,
\frac{1}{4\pi}\int_{\Mbuu{$\gamma\br\rme^{a\tb\hspace{-1.5pt}/2}$}
{$\rme^{a\tb\hspace{-1.5pt}/2}$}}\frac{1}{|\bfxb-\bfyb|}\, \Bigl(
\frac{\partial^2\eta}{\partial\yb_i\, \partial\yb_j}\, \vb_i\br
\vb_j\Bigr)(\bfyb,\tb)\; \rmd\bfyb \\ [2pt]
&& +\,
\frac{1}{4\pi}\int_{\Mbuu{$\gamma\br\rme^{a\tb\hspace{-1.5pt}/2}$}
{$\rme^{a\tb\hspace{-1.5pt}/2}$}}\frac{\xb_i-\yb_i}{|\bfxb-\bfyb|^3}\,
\Bigl( \frac{\partial\eta}{\partial\yb_i}\, \pb\Bigr)(\bfyb,\tb)\;
\rmd\bfyb \\ [2pt]
&& +\,
\frac{1}{4\pi}\int_{\Mbuu{$\gamma\br\rme^{a\tb\hspace{-1.5pt}/2}$}
{$\rme^{a\tb\hspace{-1.5pt}/2}$}}\frac{1}{|\bfxb-\bfyb|}\,
[\br\Deltab\eta\, \pb\br](\bfyb,\tb)\; \rmd\bfyb.
\end{eqnarray*}
If $\bfxb\in\Mb{3}{4}$ then
\begin{align}
|\pb_1(\bfxb,\tb)|\ &\leq\ C\int_{\Bb{2}}|\bfvb|^2\; \rmd\bfyb\
\leq\ C\int_{\Bb{1}}|\bfvb|^2\; \rmd\bfyb+C\, \biggl(
\int_{\Mb{1}{2}}|\bfvb|^s\; \rmd\bfyb\biggr)^{\! \frac{2}{s}},
\label{3.4*} \\
|\pb_3(\bfxb,\tb)|\ &\leq\ C\ \rme^{-\frac{3}{2}a\tb}
\int_{\Mbuu{$\gamma\br\rme^{a\tb\hspace{-1.5pt}/2}$}{$\rme^{a\tb
\hspace{-1.5pt}/2}$}}\bigl[\br|\bfvb|^2+|\pb|\br\bigr]\; \rmd\bfxb.
\label{3.6*}
\end{align}
Furthermore, applying the Calderon--Zygmund theorem, we obtain
\begin{equation}
\int_{\Mbu{1\hspace{2.7pt}}{$\rme^{a\tb\hspace{-1.5pt}/2}$}}|
\pb_2(\bfyb,\tb)|^\frac{s}{2}\; \rmd\bfyb\ \leq\
C\int_{\Mbu{1\hspace{2.7pt}}{$\rme^{a\tb\hspace{-1.5pt}/2}$}}
|\bfvb(\bfyb,\tb)|^s\; \rmd\bfyb. \label{3.5*}
\end{equation}
Inequalities (\ref{3.4*})--(\ref{3.5*}) imply (\ref{3.2*}).
\end{proof}

\medskip
Using Lemma \ref{L2}, we can now estimate $\cP(\delta)$ as
follows:
\begin{align}
\cP(\delta)\ &\leq\ \cc31\int_{\tb_{\delta}}^{\infty}
\biggl(\int_{\Bb{1}}|\bfvb|^2\, \rmd\bfxb\biggr)^{\!
\frac{r}{r-1}}\, \rme^{-\frac{2}{3}\, \frac{r}{r-1}\,
a(\tb-\tb_{\delta})}\; \rmd\tb \nonumber \\ \noalign{\vskip 2pt}
& \hspace{17pt} +\cc32\int_{\tb_{\delta}}^{\infty}
\biggl(\int_{\Mbu{1\hspace{2.7pt}}{$\rme^{a\tb\hspace{-1.5pt}/2}$}}
|\bfvb|^s\, \rmd\bfxb\biggr)^{\! \frac{2}{s}\, \frac{r}{r-1}}\,
\rme^{-\frac{2}{3}\,
\frac{r}{r-1}\, a(\tb- \tb_{\delta})}\; \rmd\tb \nonumber \\
\noalign{\vskip 0pt}
& \hspace{17pt} +\cc33\int_{\tb_{\delta}}^{\infty}
\biggl(\rme^{-\frac{3}{2}a\tb}\int_{\Mbuu{$\gamma\br
\rme^{a\tb\hspace{-1.5pt}/2}$}{$\rme^{a\tb\hspace{-1.5pt}/2}$}}
\bigl[\br|\bfvb|^2+|\pb|\br\bigr]\; \rmd\bfxb\biggr)^{\!
\frac{r}{r-1}}\, \rme^{-\frac{2}{3}\, \frac{r}{r-1}\,
a(\tb-\tb_{\delta})}\; \rmd\tb. \label{2.43}
\end{align}
Applying inequality (\ref{2.17*}), we deduce that the first term
on the right hand side of (\ref{2.43}) is
\begin{align}
& \leq\ \cc31\, \biggl[
{\displaystyle\mathrel{\mathop{\esssup}_{\tb>\tb_{\delta}}}}\
\int_{\Bb{1}}|\bfvb|^2\, \rmd\bfx\ \rme^{-\frac{2}{3}\br
a(\tb-\tb_{\delta})}\biggr]^{\frac{r}{r-1}-1}
\int_{\tb_{\delta}}^{\infty} \int_{\Bb{1}} |\bfvb|^2\, \rmd\bfxb\;
\rme^{-\frac{2}{3}\br a(\tb-\tb_{\delta})}\; \rmd\tb \nonumber \\
\noalign{\vskip 1pt}
& \leq\ \frac{6\cc31}{a}\ \left[ \|\varphi\br\bfv'(\, .\,
,\tb_{\delta})\|_{2;\, \Bb{4}}^2+K^I(\delta)+
K^{I\hspace{-0.3pt}I}(\delta) \right]^{\frac{r}{r-1}}.
\label{2.46}
\end{align}
The second term on the right hand hand side of (\ref{2.43}) equals
\phantom{$\cn51\cn96$}
\begin{align}
& =\ C\ \delta^{-\frac{4}{3}\, \frac{r}{r-1}}
\int_{t_0-\delta^2}^{t_0}\biggl(
\int_{\theta(t)<|\bfx-\bfx_0|<\sqrt{a}\br\rho}|\bfv|^s\;
\rmd\bfx\biggr)^{\! \frac{2}{s}\, \frac{r}{r-1}}\;
\theta^{\frac{2r}{r-1}-\frac{6}{s}\,
\frac{r}{r-1}-2+\frac{4}{3}\, \frac{r}{r-1}} (t)\; \rmd\br t \nonumber \\
\noalign{\vskip 6pt}
& =\ C\ \delta^{-\frac{4}{3}\, \frac{r}{r-1}}
\int_{t_0-\delta^2}^{t_0}\biggl(
\int_{\theta(t)<|\bfx-\bfx_0|<\sqrt{a}\br\rho}|\bfv|^s\;
\rmd\bfx\biggr)^{\! \frac{2}{s}\, \frac{r}{r-1}}\;
\theta^{\frac{r}{r-1}\br\left[\frac{2}{r}-\frac{6}{s}+
\frac{4}{3}\right]}(t)\; \rmd\br t \nonumber \\
\noalign{\vskip 6pt}
& \leq\ C\ \delta^{-\frac{4}{3}\, \frac{r}{r-1}}
\biggl[\int_{t_0-\delta^2}^{t_0}\biggl(
\int_{\theta(t)<|\bfx-\bfx_0|<\sqrt{a}\br\rho}|\bfv|^s\;
\rmd\bfx\biggr)^{\! \frac{r}{s}}\, \rmd\br
t\biggr]^{\frac{2}{r-1}}\, \biggl[\int_{t_0-\delta^2}^{t_0}
\theta^{\frac{r}{r-3}\br\left[\frac{2}{r}-\frac{6}{s}+
\frac{4}{3}\right]}
(t)\; \rmd\br t\biggr]^{\frac{r-3}{r-1}} \nonumber \\
\noalign{\vskip 6pt}
& =\ \cc51(\delta)\ \cc96(\delta), \label{2.19*}
\end{align}
where $C=C(a,\rho)$ and
\vspace{-5pt}
\begin{align*}
\cc51(\delta)\ &:=\ \biggl[\int_{t_0-\delta^2}^{t_0}\biggl(
\int_{\theta(t)<|\bfx-\bfx_0|<\sqrt{a}\br\rho}|\bfv|^s\;
\rmd\bfx\biggr)^{\! \frac{r}{s}}\, \rmd\br
t\biggr]^{\frac{2}{r-1}}\ \longrightarrow 0 \quad \mbox{for}\
\delta\to 0+, \\\noalign{\vskip 2pt}
\cc96(\delta)\ &:=\ C\ \delta^{-\frac{4}{3}\, \frac{r}{r-1}}\
\biggl[\int_{t_0-\delta^2}^{t_0}
\theta^{\frac{r}{r-3}\br\left[\frac{2}{r}-\frac{6}{s}+
\frac{4}{3}\right]} (t)\; \rmd\br t\biggr]^{\frac{r-3}{r-1}}\ =\
C\ \delta^{\frac{2r}{r-1}\,
\left[1-\frac{2}{r}-\frac{3}{s}\right]}\ =\ C.
\end{align*}
Thus, $\cc51(\delta)\ \cc96(\delta)\to 0$ as $\delta\to 0+$. Hence
the second term on the right hand hand side of (\ref{2.43}) tends
to zero as $\delta\to 0+$.

Due to the well known result from \cite{CKN}, saying that the set
of singular points of a suitable weak solution has the
$1$--dimensional Hausdorff measure equal to zero, we can assume
(without loss of generality) that $\rho$ and $\gamma$ are such
positive numbers that $\bfv$ and $p$ are bounded on the set
$\{(\bfx,t)\in\R^4;\
\sqrt{a}\br\gamma\rho<|\bfx-\bfx_0|<\sqrt{a}\br\rho$ and
$t_0-\rho^2<t<t_0\}$. Then the third term on the right hand side
of (\ref{3.2*}) is \phantom{$\cn52$}
\begin{align}
& \leq\ \cc33\int_{\tb_{\delta}}^{\infty}
\biggl(\rme^{-\frac{3}{2}a\tb}\int_{\Mbuu{$\gamma\br\rme^{a\tb
\hspace{-1.5pt}/2}$}{$\rme^{a\tb\hspace{-1.5pt}/2}$}}
\bigl[\br|\bfvb|^2+|\pb|\br\bigr]\;
\rmd\bfxb\biggr)^{\! \frac{r}{r-1}}\; \rmd\tb \nonumber \\
& =\ C\int_{t_0-\delta^2}^{t_0}\biggl(
\int_{\sqrt{a}\br\gamma\rho<|\bfx-\bfx_0|<\sqrt{a}\br\rho}
\bigl[\br|\bfv|^2+|p|\br\bigr]\; \rmd\bfx\biggr)^{\!
\frac{r}{r-1}}\, \theta^{\br 1-\frac{r}{r-1}}(t)\; \rmd\br t \nonumber \\
\noalign{\vskip 4pt}
& \leq\ C\int_{t_0-\delta^2}^{t_0}\theta^{1-\frac{r}{r-1}}(t)\;
\rmd\br t\ =:\ \cc52(\delta)\ \longrightarrow\ 0 \quad \mbox{as}\
\delta\to 0+. \label{2.45}
\end{align}

Using now (\ref{3.1*}) and estimating $\cP(\delta)$ by means of
(\ref{2.43})--(\ref{2.45}), we obtain
\begin{align*}
K^{I\hspace{-0.3pt}I}(\delta)\ & \leq\
\cc21(\delta)\ {\cP}^{\frac{r-1}{r}}(\delta) \\
&\leq \cc21(\delta)\ \Bigl\{
\Bigl(\frac{6\cc31}{a}\Bigr)^{\frac{r-1}{r}}\, \Bigl[ \|\varphi\br
\bfv'(\, .\, ,\tb_{\delta})\|_{2;\, \Bb{4}}^2+K^I(\delta)+
K^{I\hspace{-0.3pt}I}(\delta) \Bigr] \\
& \hspace{20pt} +\cc51^{\frac{r-1}{r}}(\delta)\,
\cc96^{\frac{r-1}{r}}(\delta)+\cc52^{\frac{r-1}{r}}(\delta)\Bigr\}.
\end{align*}
Assuming that $\delta$ is sufficiently small, the term
$\cc21(\delta)\, (6\cc31/a)^{\frac{r-1}{r}}\,
K^{I\hspace{-0.3pt}I}(\delta)$ on the right hand side can be
absorbed by the left hand side and we get the estimate
\phantom{$\cn55$}
\begin{align*}
K^{I\hspace{-0.3pt}I}(\delta)\ &\leq\ \cc55(\delta)\ \Bigl\{
\Bigl(\frac{6\cc31}{a}\Bigr)^{\frac{r-1}{r}}\, \Bigl[ \|\varphi\br
\bfv'(\, .\, ,\tb_{\delta})\|_{2;\, \Bb{4}}^2+K^I(\delta) \Bigr]
+\cc51^{\frac{r-1}{r}}(\delta)\, \cc96^{\frac{r-1}{r}}(\delta)+
\cc52^{\frac{r-1}{r}}(\delta) \Bigr\}\br,
\end{align*}
where
\vspace{-10pt}
\begin{displaymath}
\cc55(\delta)\ :=\ \frac{\cc21(\delta)}{1-\cc21(\delta)\,
(6\cc31/a)^{\frac{r-1}{r}}}\ \longrightarrow\ 0 \quad \mbox{for}\
\delta\to 0+.
\end{displaymath}

Thus, finally, inequality (\ref{2.17*}) yields
\phantom{$\cn56\cn57$}
\begin{align}
\|\varphi\br & \bfvb(\, .\, ,\tb)\|_{2;\, \Bb{4}}^2\,
\rme^{-\frac{2}{3}a(\tb-\tb_{\delta})}+\frac{a}{6}
\int_{\tb_{\delta}}^{\tb}\|\varphi\br\bfvb(\, .\, ,\tau)\|_{2;\,
\Bb{4}}^2\, \rme^{-\frac{2}{3}a(\tau-\tb_{\delta})}\;
\rmd\tau \nonumber \\
& \hspace{17pt} +
2\nu\int_{\tb_{\delta}}^{\tb}\|\nablab(\varphi\bfvb(\, .\,
,\tau))\|_{2;\, \Bb{4}}^2\, \rme^{-\frac{2}{3}a(\tau-
\tb_{\delta})}\; \rmd\tau \nonumber \\ \noalign{\vskip 4pt}
& \leq\ \cc56(\delta)\ \|\varphi\br\bfvb(\, .\,
,\tb_{\delta})\|_{2;\, \Bb{4}}^2+\cc57(\delta),
\label{2.46a}
\end{align}
where
\vspace{-10pt}
\begin{align}
\cc56(\delta)\ &:=\ \Bigl[ 1+\cc55(\delta)\,
\Bigl(\frac{6\cc31}{a}\Bigr)^{\frac{r-1}{r}}\Bigr]\
\longrightarrow\ 1 \quad \mbox{for}\ \delta\to 0+, \label{2.47} \\
\noalign{\vskip 2pt}
\cc57(\delta)\ &:=\ \Bigl[ 1+\cc55(\delta)\,
\Bigl(\frac{6\cc31}{a}\Bigr)^{\frac{r-1}{r}}\Bigr]\ K^I(\delta)+
\cc55(\delta)\, \bigl[ \cc51^{\frac{r-1}{r}}(\delta)\,
\cc96^{\frac{r-1}{r}}(\delta)+ \cc52^{\frac{r-1}{r}}(\delta)\bigr]
\nonumber \\ \noalign{\vskip 4pt}
& \longrightarrow\ 0 \quad \mbox{for}\ \delta\to 0+.
\label{2.48}
\end{align}

In order to control the first term on the right hand side of
(\ref{2.46a}), we shall use the next lemma.

\begin{lemma}
\label{L1}
Let function $f$ be integrable and nonnegative on the interval
$(0,\infty)$. Let $\zeta>0$. Then at least one of the two
statements holds:
\begin{alignat}{2}
\hspace{-7pt}  {\rm (A)} \hspace{29.4mm} &
\int_{\sigma}^{\infty}f(\tau)\;
\rme^{-\frac{2}{3}a(\tau-\sigma)}\; \rmd\tau\ \longrightarrow\ 0
\quad \mbox{for}\ \sigma\to\infty, \hspace{29mm} \label{0.2a}
\end{alignat}

\begin{list}{}
{\setlength{\topsep 4pt}
\setlength{\itemsep 6pt}
\setlength{\leftmargin 25pt}
\setlength{\rightmargin 0pt}
\setlength{\labelwidth 19pt}}

\item[{\rm (B)} ]
there exists a set $\cE'_{\zeta}\subset(0,\infty)$ such that
$m_1\bigl(\cE'_{\zeta}\cap(\sigma,\infty)\bigr)>0$ for each
$\sigma>0$ (where $m_1$ denotes the $1$-dimensional Lebesgue
measure) and
\begin{equation}
f(\sigma)\ \leq\ \frac{2a(1+\zeta)}{3}
\int_{\sigma}^{\infty}f(\tau)\; \rme^{-\frac{2}{3}a(\tau-
\sigma)}\; \rmd\tau \qquad \mbox{for}\ \sigma\in\cE'_{\zeta}.
\label{0.2}
\end{equation}
\end{list}
\end{lemma}

\begin{proof}
Denote $h(\sigma):=\int_{\sigma}^{\infty}f(\tau)\;
\rme^{-\frac{2}{3}a\tau}\; \rmd\tau$. Then
$h'(\sigma)=-f(\sigma)\, \rme^{-\frac{2}{3}a\sigma}$ for
a.a.~$\sigma\in(0,\infty)$. Inequality (\ref{0.2}) is equivalent
to $h'(\sigma)+\frac{2}{3}\br a(1+\zeta)\, h(\sigma)\geq 0$, which
is further equivalent to $[\rme^{\frac{2}{3}a(1+\zeta)\sigma}\,
h(\sigma)]'\geq 0$.

Assume that statement (B) does not hold. Then there exists
$\sigma_0>0$ such that $[\rme^{\frac{2}{3}a(1+\zeta)\sigma}\,
h(\sigma)]' \\ <0$ for a.a.~$\sigma\in(\sigma_0,\infty)$. Hence
$\rme^{\frac{2}{3}a(1+\zeta)\sigma}\,
h(\sigma)<\rme^{\frac{2}{3}a(1+\zeta)\sigma_0}\, h(\sigma_0)$,
i.e.
\begin{displaymath}
\rme^{\frac{2}{3}a\sigma}\, h(\sigma)\ <\
\rme^{-\frac{2}{3}a\zeta\sigma}\,
\rme^{\frac{2}{3}a(1+\zeta)\sigma_0}\, h(\sigma_0)
\end{displaymath}
for a.a.~$\sigma\in(\sigma_0,\infty)$. Thus, statement (A) holds.
The proof of Lemma \ref{L1} is completed.
\end{proof}

\vspace{5pt}
If we apply Lemma \ref{L1} with $\sigma=\tb_{\delta}$ and
$f(\tb_{\delta})=\|\varphi\bfvb(\, .\, ,\tb_{\delta})\|_{2;\,
\Bb{4}}^2$, we obtain that either
\begin{equation}
\int_{\tb_{\delta}}^{\infty} \|\varphi\bfvb(\, .\, ,\tau)\|_{2;\,
\Bb{4}}^2\; \rme^{-\frac{2}{3}a(\tau-\tb_{\delta})}\; \rmd\tau\
\longrightarrow\ 0 \quad \mbox{for}\ \tb_{\delta}\to\infty,
\label{0.4}
\end{equation}
or there exists a set $\cE'_{\zeta}\subset(0,\infty)$ with the
properties named in item (B) of Lemma \ref{L1} such that
\begin{equation}
\|\varphi\bfvb(\, .\, ,\tb_{\delta})\|_{2;\, \Bb{4}}^2\ \leq\
\frac{2a(1+\zeta)}{3}\int_{\tb_{\delta}}^{\infty}
\|\varphi\bfvb(\, .\, ,\tau)\|_{2;\, \Bb{4}}^2\;
\rme^{-\frac{2}{3}a(\tau-\tb_{\delta})}\; \rmd\tau \qquad
\mbox{for}\ \tb_{\delta}\in\cE'_{\zeta}. \label{0.3}
\end{equation}

%----------------------------------------------------------------------

\section{Completion of the proof of Theorem \ref{T1} in the case of
(\ref{0.4})} \label{S4}

In this section, we assume that (\ref{0.4}) holds. Then there
exists a set $\cG'\subset(0,\infty)$ such that
$m_1(\cG'\cap(\sigma,\infty))>0$ for each $\sigma>0$ and
\begin{equation}
\|\varphi\bfvb(\, .\, ,\tb_{\delta})\|_{2;\, \Bb{4}}\
\longrightarrow\ 0 \qquad \mbox{for}\ \tb_{\delta}\in\cG,\
\tb_{\delta}\to\infty. \label{4.1}
\end{equation}
(This can be easily proven by contradiction.) Denote by $\cG$ the
set of $\delta>0$, corresponding to $\tb_{\delta}\in\cG'$, where
$\delta$ and $\tb_{\delta}$ are related by formula (\ref{2.6a*}).
Inequality (\ref{2.46a}) yields \phantom{$\cn68$}
\begin{gather}
\|\varphi\br\bfvb(\, .\, ,\tb)\|_{2;\, \Bb{4}}^2\,
\rme^{-\frac{2}{3}a(\tb-\tb_{\delta})}+\frac{a}{6}
\int_{\tb_{\delta}}^{\tb}\|\varphi\br\bfvb(\, .\, ,\tau)\|_{2;\,
\Bb{4}}^2\, \rme^{-\frac{2}{3}a(\tau-\tb_{\delta})}\;
\rmd\tau \nonumber \\
+\, 2\nu\int_{\tb_{\delta}}^{\tb}\|\nablab(\varphi\br\bfvb(\, .\,
,\tau))\|_{2;\, \Bb{4}}^2\, \rme^{-\frac{2}{3}a(\tau-
\tb_{\delta})}\; \rmd\tau\ \leq\ \cc68(\delta) \label{4.2}
\end{gather}
for $\tb_{\delta}\in\cG'$ and $\tb>\tb_{\delta}$, where
$\cc68(\delta)\to 0$ as $\delta\to 0$, $\delta\in\cG$. Applying
now (\ref{4.2}), we can estimate the integrals on the right hand
side of (\ref{2.13*}) in the case when $\delta\in\cG$:
\begin{align}
\int_{\tb_{\delta}}^{\infty} & \|\nablab(\varphi\br\bfvb)\|_{2;\,
\Bb{4}}^2\, \rme^{-\frac{2}{3}a\br(\tb-\tb_{\delta})}\;
\rmd\br\tb\ \leq\ \frac{1}{2\nu}\, \cc68(\delta), \label{2.14} \\
\noalign{\vskip 6pt}
\int_{\tb_{\delta}}^{\infty} & \|\varphi\bfvb\|_{2;\, \Bb{4}}^6\,
\rme^{-2a\br(\tb-\tb_{\delta})}\; \rmd\br\tb \nonumber \\
\noalign{\vskip -4pt}
& \leq\ \Bigl[{\displaystyle
\mathrel{\mathop{\esssup}_{\tb>\tb_{\delta}}}}\
\bigl(\|\varphi\br\bfvb(\, .\, ,\tb)\|_{2;\, \Bb{4}}^4\,
\rme^{-\frac{4}{3}a(\tb-\tb_{\delta})}\, \bigr)\Bigr]\
\int_{\tb_n}^{\infty}\|\varphi\bfvb\|_{2;\, \Bb{4}}^2\,
\rme^{-\frac{2}{3}a\br(\tb-\tb_{\delta})}\; \rmd\br\tb \nonumber \\
\noalign{\vskip 2pt}
& \leq\ \frac{6}{a}\, \cc68^3(\delta). \label{2.33}
\end{align}
Since the right hand sides of (\ref{2.14}) and (\ref{2.33}) tend
to zero for $\delta\to 0$, $\delta\in\cG$, we obtain (by means of
estimate (\ref{2.13*})) that $G^{I\hspace{-0.3pt}I}(\delta)\to 0$
for the same $\delta$. This, together with (\ref{2.4*}), proves
(\ref{2.3*}). Hence $(\bfx_0,t_0)$ is a regular point of solution
$\bfv$.

%----------------------------------------------------------------------

\section{Completion of the proof of Theorem \ref{T1} in the case of
(\ref{0.3})} \label{S5}

In this section, we assume that (\ref{0.3}) holds.

\vspace{4pt} \noindent
{\bf A partition of function $\varphi$.} \ Let $\xi\in(0,1)$.
Function $\varphi$ can be expressed in the form $\varphi_1^{\xi}+
\varphi_2^{\xi}$, where both the functions $\varphi_1^{\xi}$ and
$\varphi_2^{\xi}$ are in $C_0^{\infty}(\R^3)$,
\begin{displaymath}
\varphi_1^{\xi}(\bfxb)\ \ \left\{ \begin{array}{ll} =1 &
\mbox{for}\
|\bfxb|\leq 1+\frac{1}{4}\br\xi, \\ [4pt] \in[0,1] & \mbox{for}\
1+\frac{1}{4}\br\xi<|\bfxb|\leq 1+\frac{3}{4}\br\xi, \\
[4pt] =0 & \mbox{for}\ 1+\frac{3}{4}\br\xi<|\bfxb|, \end{array}
\right. \qquad \varphi_2^{\xi}(\bfxb)\ :=\
\varphi(\bfxb)-\varphi^{\xi}(\bfxb).
\end{displaymath}
Function $\varphi_1^{\xi}$ can be chosen so that
$|\nabla'\varphi_1^{\xi}|\leq 4\, \xi^{-1}$.

\vspace{4pt} \noindent
{\bf A Friedrichs--type estimate of $\varphi_1^{\xi}\br\bfvb$ in
$\Bb{1+\xi}$.} \ Applying the so called Bogovskij operator in
$\Mb{1}{1+\xi}$, one can construct a function
$\bfw^{\xi}\in\bfW^{1,2}_0(\Mb{1}{1+\xi})$ such that ${\rm
div}'\br\bfw^{\xi}=-\nablab\varphi_1^{\xi}\cdot\bfvb$ and
$\|\nabla\bfw^{\xi}\|_{2;\, \Mb{1}{1+\xi}}\leq C\br\xi^{-1}\,
\|\bfvb\|_{2;\, \Mb{1}{1+\xi}}$. If we extend function
$\bfw^{\xi}$ by zero to the whole ball $\Bb{1+\xi}$ then
$\varphi_1^{\xi}\bfvb-\bfw^{\xi}$ is divergence--free and belongs
to $\bfW^{1,2}_0(\Bb{1+\xi})$. Now we have
\begin{equation}
\|\varphi_1^{\xi}\bfvb-\bfw^{\xi}\|_{2;\, \Bb{1+\xi}}\ \leq\
\frac{1+\xi}{\sqrt{\lambda_S(B_1)}}\
\bigl\|\nablab(\varphi_1^{\xi}\bfvb- \bfw^{\xi})\bigr\|_{2;\,
\Bb{1+\xi}}\br. \label{5.9*}
\end{equation}
(Recall the $\lambda_S(B_1)$ is the first eigenvalue of the
Dirichlet--Stokesian and $\pi^2$ is the principal eigenvalue of
the Dirichlet--Laplacian in the unit ball -- see Section
\ref{S1}.) Hence
\begin{align*}
\|\varphi_1^{\xi}\bfvb\|_{2;\, \Bb{1+\xi}}\ &\leq\
\|\varphi_1^{\xi}\bfvb-\bfw^{\xi}\|_{2;\, \Bb{1+\xi}}+
\|\bfw^{\xi}\|_{2;\, \Bb{1+\xi}} \nonumber \\ \noalign{\vskip 2pt}
&\leq\ \frac{1+\xi}{\sqrt{\lambda_S(B_1)}}\
\bigl\|\nablab(\varphi_1^{\xi}\bfvb- \bfw^{\xi})\bigr\|_{2;\,
\Bb{1+\xi}}+\frac{1+\xi}{\pi}\,
\|\nablab\bfw^{\xi}\|_{2;\, \Bb{1+\xi}} \nonumber \\
\noalign{\vskip 2pt}
&\leq\ \frac{1+\xi}{\sqrt{\lambda_S(B_1)}}\
\bigl\|\nablab(\varphi_1^{\xi}\bfvb) \bigr\|_{2;\,
\Bb{1+\xi}}+\frac{2\br(1+\xi)}{\pi}\
\|\nablab\bfw^{\xi}\|_{2;\, \Bb{1+\xi}} \nonumber \\
\noalign{\vskip 2pt}
&=\ \frac{1+\xi}{\sqrt{\lambda_S(B_1)}}\
\bigl\|\nablab(\varphi_1^{\xi}\bfvb) \bigr\|_{2;\,
\Bb{1+\xi}}+\frac{2\br(1+\xi)}{\pi}\
\|\nablab\bfw^{\xi}\|_{2;\, \Mb{1}{1+\xi}} \nonumber \\
\noalign{\vskip 2pt}
&\leq\ \frac{1+\xi}{\sqrt{\lambda_S(B_1)}}\
\bigl\|\nablab(\varphi_1^{\xi}\bfvb) \bigr\|_{2;\,
\Bb{1+\xi}}+\frac{C\br(1+\xi)}{\xi\pi}\ \|\bfvb\|_{2;\,
\Mb{1}{1+\xi}}\br.
\end{align*}
This implies that to each $\xi>0$ and $\kappa>0$ there exists
$\cn67(\kappa,\zeta)>0$ such that
\begin{equation}
\|\varphi_1^{\xi}\bfvb\|_{2;\, \Bb{1+\xi}}^2\ \leq\
\frac{(1+\kappa)\, (1+\xi)^2}{\lambda_S(B_1)}\
\bigl\|\nablab(\varphi_1^{\xi}\bfvb) \bigr\|_{2;\,
\Bb{1+\xi}}^2+\cc67(\kappa,\xi)\, \|\bfvb\|_{2;\,
\Mb{1}{1+\xi}}^2\br. \label{5.1*}
\end{equation}

\vspace{4pt} \noindent
{\bf Application of inequality (\ref{0.3}).} \ The integrand
$\|\varphi\br\bfvb\|_{2;\, \Bb{4}}^2$ on the right hand side of
(\ref{0.3}) can be expressed in the form
\begin{equation}
\|\varphi\br\bfvb\|_{2;\, \Bb{4}}^2\ =\
\|\varphi_1^{\xi}\br\bfvb\|_{2;\, \Bb{1+\xi}}^2+
2\br\bigl(\varphi_1^{\xi}\br\bfvb,\,
\varphi_2^{\xi}\br\bfvb\bigr)_{2;\, \Mb{1}{1+\xi}}+
\|\varphi_2^{\xi}\br\bfv'\|_{2;\, \Mb{1}{4}}^2\br,
\label{5.2*}
\end{equation}
where \phantom{$\cn02$}
\begin{displaymath}
2\int_{\tb_{\delta}}^{\infty}\bigl[\br\bigl(\varphi_1^{\xi}\br\bfvb,\,
\varphi_2^{\xi}\br\bfvb\bigr)_{2;\, \Mb{1}{1+\xi}}+
\|\varphi_2^{\xi}\br\bfv'\|_{2;\, \Mb{1}{4}}^2\bigr]\;
\rme^{-\frac{2}{3}a(\tau-\tb_{\delta})}\; \rmd\tau\ =:\
\cc02(\delta,\xi)\ \longrightarrow\ 0
\end{displaymath}
for $\delta\to 0+$ for each fixed $\xi$ due to Lemma \ref{L4}.
Substituting inequality (\ref{0.3}), with the integrand on the
right hand side expressed as in (\ref{5.2*}) to (\ref{2.46a}), we
obtain \phantom{$\cn03$}
\begin{align}
\| & \varphi\br\bfvb(\, .\, ,\tb)\|_{2;\, \Bb{4}}^2\,
\rme^{-\frac{2}{3}a(\tb-\tb_{\delta})}+\mu\int_{\tb_{\delta}}^{\infty}
\|\varphi_1^{\xi}\br\bfvb(\, .\, ,\tau)\|_{2;\, \Bb{1+\xi}}^2\;
\rme^{-\frac{2}{3}a(\tau-\tb_{\delta})}\; \rmd\tau \nonumber \\
& \hspace{17pt} +
2\nu\int_{\tb_{\delta}}^{\tb}\|\nablab(\varphi\br\bfvb(\, .\,
,\tau))\|_{2;\, \Bb{4}}^2\,
\rme^{-\frac{2}{3}a(\tau-\tb_{\delta})}\; \rmd\tau \nonumber \\
\noalign{\vskip 0pt}
& \leq\ \cc03(\delta,\zeta,\mu)\,
\biggl[\int_{\tb_{\delta}}^{\infty} \|\varphi_1^{\xi}\br\bfvb(\,
.\, ,\tau)\|_{2;\, \Bb{1+\xi}}^2\;
\rme^{-\frac{2}{3}a(\tau-\tb_{\delta})}\;
\rmd\tau+\cc02(\delta,\xi)\biggr]+\cc57(\delta),
\label{5.4*}
\end{align}
where $\mu$ is an arbitrary positive number and
\begin{equation}
\cc03(\delta,\zeta,\mu)\ :=\ \frac{2a(1+\zeta)}{3}\,
\cc56(\delta)-\frac{a}{6}+\mu\ \longrightarrow\
\frac{a}{2}+\mu+\frac{2a\zeta}{3} \qquad \mbox{as}\ \delta\to 0+
\label{5.6*}
\end{equation}
(due to (\ref{2.47})). Inequality (\ref{5.4*}) is satisfied for
$\tb_{\delta}\in\cE'_{\zeta}$, for corresponding $\delta=\rho\,
\rme^{-\frac{1}{2}a\tb_{\delta}}$ (we denote by $\cE_{\zeta}$ the
set of these $\delta$) and for all $\tb>\tb_{\delta}$. The
integrand $\|\varphi_1^{\xi}\br\bfvb(\, .\, ,\tau)\|_{2;\,
\Bb{1+\xi}}^2$ in the integral on the right hand side of
(\ref{5.4*}) can be further estimated by means of inequality
(\ref{5.1*}). The term $\|\nablab(\varphi_1^{\xi}\bfvb)\|_{2;\,
\Bb{1+\xi}}^2$ on the right hand side of (\ref{5.1*}) is estimated
as follows:
\begin{align*}
\|\nablab(\varphi_1^{\xi} & \bfvb)\|_{2;\, \Bb{1+\xi}}^2\ =\
\|\varphi_1^{\xi}\br\nablab\bfvb\|_{2;\, \Bb{1+\xi}}^2+2\, \bigl(
\varphi_1^{\xi}\br\nablab\bfvb,\,
\nablab\varphi_1^{\xi}\otimes\bfvb\bigr)_{2;\, \Bb{1+\xi}}+
\|\nablab\varphi_1^{\xi}\otimes\bfvb\|_{2;\, \Bb{1+\xi}}^2
\\ \noalign{\vskip 2pt}
& \leq\ \|\nablab\bfvb\|_{2;\, \Bb{1+\xi}}^2+\int_{\Mb{1}{1+\xi}}
\nablab(\varphi_1^{\xi})^2\cdot\nablab\bfvb\cdot\bfvb\; \rmd\bfxb+
\|\nablab\varphi_1^{\xi}\otimes\bfvb\|_{2;\, \Mb{1}{1+\xi}}^2
\\ \noalign{\vskip 2pt}
& \leq\ \|\nablab(\varphi\br\bfvb)\|_{2;\,
\Bb{4}}^2-\frac{1}{2}\int_{\Mb{1}{1+\xi}}\Deltab(\varphi_1^{\xi})^2\,
|\bfvb|^2\; \rmd\bfxb+
\|\nablab\varphi_1^{\xi}\otimes\bfvb\|_{2;\, \Mb{1}{1+\xi}}^2\br.
\end{align*}
Now, inequalities (\ref{5.1*}) and (\ref{5.4*}) yield
\phantom{$\cn05$}
\begin{align}
\| & \varphi\br\bfvb(\, .\, ,\tb)\|_{2;\, \Bb{4}}^2\,
\rme^{-\frac{2}{3}a(\tb-\tb_{\delta})}+\mu\int_{\tb_{\delta}}^{\infty}
\|\varphi_1^{\xi}\br\bfvb(\, .\, ,\tau)\|_{2;\, \Bb{1+\xi}}^2\;
\rme^{-\frac{2}{3}a(\tau-\tb_{\delta})}\; \rmd\tau \\
& \hspace{17pt} +
2\nu\int_{\tb_{\delta}}^{\tb}\|\nablab(\varphi\br\bfvb(\, .\,
,\tau))\|_{2;\, \Bb{4}}^2\,
\rme^{-\frac{2}{3}a(\tau-\tb_{\delta})}\; \rmd\tau \nonumber \\
\noalign{\vskip 0pt}
& \leq\ \cc03(\delta,\zeta,\mu)\, \frac{(1+\kappa)\,
(1+\xi)^2}{\lambda_S(B_1)} \int_{\tb_{\delta}}^{\infty}
\|\nablab(\varphi\br\bfvb(\, .\, ,\tau))\|_{2;\, \Bb{4}}^2\;
\rme^{-\frac{2}{3}a(\tau-\tb_{\delta})}\; \rmd\tau \nonumber \\
\noalign{\vskip 4pt}
& \hspace{17pt} +\br\cc03(\delta,\zeta,\mu)\ \frac{(1+\kappa)\,
(1+\xi)^2}{\lambda_S(B_1)}\
\cc05(\delta,\kappa,\xi)+\cc03(\delta,\zeta,\mu)\
\cc02(\delta,\xi)+\cc57(\delta), \label{5.5*}
\end{align}
where
\begin{align*}
\cc05(\delta,\kappa,\xi)\ &:=\
\int_{\tb_{\delta}}^{\infty}\Bigl[\br
\frac{1}{2}\int_{\Mb{1}{1+\xi}}|\Deltab(\varphi_1^{\xi})^2|\,
|\bfvb|^2\; \rmd\bfxb+ \|\nablab\varphi_1^{\xi}
\otimes\bfvb\|_{2;\, \Mb{1}{1+\xi}}^2\Bigr]\;
\rme^{-\frac{2}{3}a(\tb-\tb_{\delta})}\; \rmd\br\tb \\
\noalign{\vskip 2pt}
& \hspace{20pt} +\cc67(\kappa,\xi)\int_{\tb_{\delta}}^{\infty}
\|\bfvb\|_{2;\, \Mb{1}{1+\xi}}^2\;
\rme^{-\frac{2}{3}a(\tb-\tb_{\delta})}\; \rmd\br\tb.
\end{align*}
Due to (\ref{2.11*}), $\cc05(\delta,\kappa,\xi)\to 0$ for all
fixed $\kappa$, $\xi$ and $\delta\to 0+$. We observe from
(\ref{5.6*}) and from the inequality $a<4\nu\lambda_S(B_1)$ (see
the assumptions of Theorem \ref{T1}) that there exist positive
numbers $\delta_0$, $\xi$, $\zeta$, $\mu$, $\kappa$ and $\epsilon$
such that
\vspace{-8pt}
\begin{equation}
2\nu-\cc03(\delta,\zeta,\mu)\, \frac{(1+\kappa)\,
(1+\xi)^2}{\lambda_S(B_1)}\ \geq\ \epsilon
\label{5.8*}
\end{equation}
for all $0<\delta\leq\delta_0$. Then inequality (\ref{5.5*}), with
these fixed numbers $\xi$, $\zeta$, $\mu$ $\kappa$ and $\epsilon$,
and with $\delta\in\cE_{\zeta}$, $0<\delta\leq\delta_0$, yields
\begin{align}
\|\varphi\br\bfvb(\, .\, & ,\tb)\|_{2;\, \Bb{4}}^2\
\rme^{-\frac{2}{3}a(\tb-\tb_{\delta})}+\mu\int_{\tb_{\delta}}^{\infty}
\|\varphi_1^{\xi}\br\bfvb(\, .\, ,\tau)\|_{2;\, \Bb{1+\xi}}^2\;
\rme^{-\frac{2}{3}a(\tau-\tb_{\delta})}\; \rmd\tau \\
& \hspace{17pt} +
\epsilon\int_{\tb_{\delta}}^{\tb}\|\nablab(\varphi\br\bfvb(\, .\,
,\tau))\|_{2;\, \Bb{4}}^2\,
\rme^{-\frac{2}{3}a(\tau-\tb_{\delta})}\; \rmd\tau \nonumber \\
\noalign{\vskip 2pt}
& \leq\ \cc03(\delta,\zeta,\mu)\ \frac{(1+\kappa)\,
(1+\xi)^2}{\lambda_S(B_1)}\
\cc05(\delta,\kappa,\xi)+\cc03(\delta,\zeta,\mu)\
\cc02(\delta,\xi)+ \cc57(\delta). \label{5.7*}
\end{align}
Thus, we deduce that an analogous expression to the left hand side
of (\ref{4.2}) tends to zero as $\delta\to 0$,
$\delta\in\cE_{\zeta}$. The proof of Theorem \ref{T1} can now be
completed in the same way as in Section \ref{S4} after
(\ref{4.2}).

\section{A generalization of Theorem \ref{T1}} \label{S6}

The assumption $a<4\nu\lambda_S(B_1)$ in Theorem \ref{T1}
represents a restriction on the shape of paraboloid $P_a$: the
paraboloid cannot be arbitrarily wide and set $\Uar$ (where $\bfv$
is supposed to satisfy the Serrin--type condition, considered for
fixed $\rho$) therefore cannot be arbitrarily small. The condition
$a<4\nu\lambda_S(B_1)$ is used only in Section \ref{S5}, where it
guarantees the validity of inequality (\ref{5.8*}). There arises a
natural question whether Theorem \ref{T1} can be improved so that
the Serrin--type integrability condition would be assumed on a
smaller set than $\Uar$. We present such a possibility in this
section.

In order to stress the dependence on parameter $a$, we further
denote the function $\theta(t)\equiv\sqrt{a(t_0-t)}$ by
$\theta_a(t)$. Then $\theta_1(t)=\sqrt{t_0-t}$. Note that set
$\Uar$ (see Section \ref{S1}) can also be defined as follows:
\begin{displaymath}
\Uar\ =\ \bigl\{\, (\bfx,t)\in\R^4;\ t_0-\rho^2<t<t_0,\ \bfx\in
B_{\sqrt{a}\br\rho}(\bfx_0)\smallsetminus\overline{[\bfx_0+
\theta_1(t)\, B_{\sqrt{a}}(\bfzero)]}\, \bigr\}. \label{6.1}
\end{displaymath}
Furthermore, since the least eigenvalue of the Dirichlet--Stokes
operator in $B_{\sqrt{a}}$ (any ball in $\R^3$ with the radius
$\sqrt{a}$) is $\lambda_S(B_{\sqrt{a}})=\lambda_S(B_1)/a$, the
condition $a<4\nu\lambda_S(B_1)$ is equivalent to
$1<4\nu\lambda_S(B_{\sqrt{a}})$.

These notes lead us to the generalization of Theorem \ref{T1}:
assume that $D$ is a bounded open set in $\R^3$ (not necessarily
connected), with a Lipschitzian boundary and containing point
$\bfzero$. Let $\lambda_S(D)$ be the least eigenvalue of the
Dirichlet--Stokes operator in $D$. We define set $U_{\rho}$
(analogous to the previous $\Uar$) to be
\begin{displaymath}
U_{\rho}\ :=\ \bigl\{\, (\bfx,t)\in\R^4;\ t_0-\rho<t<t_0,\ \bfx\in
B_{\rho}(\bfx_0)\smallsetminus\overline{[\bfx_0+\theta_1(t)\,
D]}\, \bigr\}.
\end{displaymath}
Now we can formulate the theorem:

\begin{theorem}
\label{T2}
Let $\bfv$ be a suitable weak solution of system (\ref{1.1}),
(\ref{1.2}), $(\bfx_0,t_0)\in Q_T$ and $\rho>0$ be so small that
$Q_{1,\rho}\subset Q_T$. Assume that $D$ is a bounded open set in
$\R^3$ with a Lipschitzian boundary, containing point $\bfzero$,
such that $1<4\nu\lambda_S(D)$ and function $\bfv$ satisfies the
integrability condition in set $U_{\rho}$\br:
\begin{equation}
\int_{t_0-\rho^2}^{t_0}\biggl(\int_{
B_{\rho}(\bfx_0)\smallsetminus[\bfx_0+\theta_1(t)\, D]}
|\bfv(\bfx,t)|^s\; \rmd\bfx\biggr)^{\! \frac{r}{s}}\, \rmd\br t\
<\ \infty \label{6.1*}
\end{equation}
for some $r$, $s$, satisfying inequalities (\ref{1.7*}). Then
$(\bfx_0,t_0)$ is a regular point of solution $\bfv$.
\end{theorem}

\noindent
Theorem \ref{T2} can be proven in the same way as Theorem
\ref{T1}, up to smaller modifications. The most important ones
are: we use function $\theta_1$ instead of $\theta_a$, we obtain
the sets
\begin{align*}
V'\ &=\ \bigl\{(\bfxb,\tb)\in\R^4;\ \tb>0\ \mbox{and}\ \bfxb\in
D\bigr\}, \\ \noalign{\vskip 4pt}
U'\ &=\ \bigl\{(\bfxb,\tb)\in\R^4;\ \tb>0,\
|\bfxb|<\rme^{\frac{1}{2}\tb},\ \bfxb\not\in D\bigr\}
\end{align*}
instead of $\Var$ and $\Uar$, we deal with $U_{\xi}(D)$ (the
$\xi$--neighbourhood of $D$) instead of $\Bb{1+\xi}$ and the
cut--off function $\varphi$ decreases from one to zero in
$\Mb{R+1}{R+2}$ instead of $\Mb{3}{4}$ (where $R$ is so large that
$D\subset B_R(\bfzero)$).

Theorem \ref{T1} is a special case of Theorem \ref{T2},
corresponding to the choice $D=B_{\sqrt{a}}(\bfzero)$. However,
$D$ can generally have another shape than $B_{\sqrt{a}}(\bfzero)$
(it can be e.g.~``larger'' than $B_{\sqrt{a}}(\bfzero)$ in some
directions) and it can still satisfy the condition
$1<4\nu\lambda_S(D)$. The question of dependence of $\lambda_S(D)$
on $D$ is discussed in greater detail in \cite{SeVe} and
\cite{Yu}.

%---------------------------------------------------------------
\vspace{5pt} \noindent
{\bf Acknowledgments.} \ The research was supported by the Grant
Agency of the Czech Republic (grant No.~13-00522S) and by the
Academy of Sciences of the Czech Republic (RVO 67985840).

%----------------------------------------------------------------------

\noindent
\begin{tabular}{ll}
{\it Author's address:} \ & Institute of Mathematics, Academy of
Sciences of the Czech Republic \\ & \v{Z}itn\'a 25, 115 67 Praha 1
\\ & Czech Republic \\ & {\it e-mail:} \ neustupa@math.cas.cz
\end{tabular}

\end{document}